\documentclass[11pt]{amsart}
\usepackage[top=1.5in, bottom=1.5in, left=1.5in, right=1.5in]{geometry}
\expandafter\let\csname ver@amsthm.sty\endcsname\relax

\usepackage{amsmath}
\usepackage{amssymb}
\usepackage{hyperref}
\usepackage{amsthm}
\usepackage{mathtools}
\usepackage[capitalize,noabbrev]{cleveref}
\usepackage{xcolor}
\usepackage{tikz}

\allowdisplaybreaks

\numberwithin{equation}{section}

\newtheorem{thm}{Theorem}[section]
\newtheorem{lemma}[thm]{Lemma}
\newtheorem{cor}[thm]{Corollary}
\newtheorem{prop}[thm]{Proposition}

\newtheorem{Definition}[thm]{Definition}

\newtheorem{Example}[thm]{Example}
\newenvironment{example}
  {\begin{Example}\rm}{\end{Example}}

\newtheorem{Remark}[thm]{Remark}
\newenvironment{remark}
  {\begin{Remark}\rm}{\end{Remark}}
  
  \newtheorem{Question}[thm]{Question}
\newenvironment{question}
  {\begin{Question}\rm}{\end{Question}}
  
    \newtheorem{Problem}[thm]{Problem}
\newenvironment{problem}
  {\begin{Problem}\rm}{\end{Problem}}
  
\crefname{thm}{Theorem}{Theorems}
\crefname{lemma}{Lemma}{Lemmas}
\crefname{cor}{Corollary}{Corollaries}
\crefname{prop}{Proposition}{Propositions}

\crefname{definition}{Definition}{Definitions}
\crefname{example}{Example}{Examples}
\crefname{remark}{Remark}{Remarks}
\crefname{question}{Question}{Questions}
\crefname{problem}{Problem}{Problems}

\newcommand{\emailhref}[1]{\email{\href{#1}{#1}}}
\newcommand{\dfn}[1]{\textcolor{blue}{\emph{#1}}}

\DeclareRobustCommand{\qbinom}{\genfrac{\lbrack}{\rbrack}{0pt}{}}

\title[Upho lattices I]{Upho lattices I: \\ examples and non-examples of cores}
\author{Sam Hopkins}\emailhref{samuelfhopkins@gmail.com}
\address{Department of Mathematics, Howard University, Washington, DC 20059}

\begin{document}

\dedicatory{Dedicated to Richard Stanley on the occasion of his 80th birthday}

\begin{abstract}
A poset is called upper homogeneous, or ``upho,'' if every principal order filter of the poset is isomorphic to the whole poset. We study (finite type $\mathbb{N}$-graded) upho lattices, with an eye towards their classification. Any upho lattice has associated to it a finite graded lattice called its core, which determines its rank generating function. We investigate which finite graded lattices arise as cores of upho lattices, providing both positive and negative results. On the one hand, we show that many well-studied finite lattices do arise as cores, and we present combinatorial and algebraic constructions of the upho lattices into which they embed.  On the other hand, we show there are obstructions which prevent many finite lattices from being cores.
\end{abstract}

\maketitle

\section{Introduction} \label{sec:intro}

Symmetry is a fundamental theme in mathematics. A close cousin of symmetry is self-similarity, where a part resembles the whole. In this paper, we study certain partially ordered sets that are self-similar in a precise sense. Namely, a poset is called \dfn{upper homogeneous}, or ``\dfn{upho},'' if every principal order filter of the poset is isomorphic to the whole poset. In other words, a poset $\mathcal{P}$ is upho if, looking up from each element $p \in \mathcal{P}$, we see another copy of $\mathcal{P}$. Upho posets were introduced recently by Stanley~\cite{stanley2020upho, stanley2024theorems}. We believe they are a natural and rich class of posets which deserve further attention.

Upho posets are infinite. In order to be able to apply the tools of enumerative and algebraic combinatorics, we need to impose some finiteness condition on the posets we consider. Thus, we restrict our attention to finite type $\mathbb{N}$-graded posets. These infinite posets $\mathcal{P}$ possess a rank function $\rho\colon \mathcal{P}\to\mathbb{N}$ for which we can form the rank generating function
\[ F(\mathcal{P}; x) \coloneqq \sum_{p \in \mathcal{P}} x^{\rho(p)}.\] 
Henceforth, upho posets are assumed finite type $\mathbb{N}$-graded unless otherwise specified.

The big problem concerning upho posets that we are interested in is the following.

\begin{problem} \label{prob:classify}
Classify upho lattices.
\end{problem}

\Cref{prob:classify} is likely a hard problem, perhaps even impossible. But let us explain why there is some hope of making progress on this problem. It was shown by Gao, Guo, Seetharaman, and Seidel~\cite{gao2020upho} that there are uncountably many different rank generating functions of (finite type $\mathbb{N}$-graded) upho posets. This prevents us from being able to say much about upho posets in general. However, the situation is different for \emph{lattices}: in~\cite{hopkins2022note} we showed that the rank generating function of an upho lattice is the multiplicative inverse of a polynomial with integer coefficients.

More precisely, we made the following observation about the rank generating function of an upho lattice. Let $\mathcal{L}$ be an upho lattice, and let $L \coloneqq [\hat{0},s_1 \vee \cdots \vee s_r] \subseteq \mathcal{L}$ denote the interval in $\mathcal{L}$ from its minimum $\hat{0}$ to the join of its atoms $s_1, \ldots, s_r$. We refer to the finite graded lattice $L$ as the \dfn{core} of the upho lattice $\mathcal{L}$. We showed  in~\cite{hopkins2022note} that 
\begin{equation} \label{eqn:rank_char}
F(\mathcal{L}; x) = \chi^*(L; x)^{-1}, 
\end{equation}
where $\chi^*(L; x) = \sum_{p \in L} \mu(\hat{0},p) x^{\rho(p)}$ is the (reciprocal) characteristic polynomial of~$L$. In this way, the core of an upho lattice determines its rank generating function.\footnote{In fact, since the flag $f$-vector of any upho poset is determined by its rank generating function (see~\cite[\S3]{stanley2024theorems}),  the core of an upho lattice determines its entire flag $f$-vector.}

The core does not determine the upho lattice completely. That is, there are different upho lattices with the same core. Nevertheless, to resolve \cref{prob:classify} we would certainly need to answer the following question.

\begin{question} \label{question:main}
Which finite graded lattices are cores of upho lattices?
\end{question}

\Cref{question:main} can be thought of as a kind of tiling problem: our goal is to tile an infinite, fractal lattice $\mathcal{L}$ using copies of some fixed finite lattice $L$, or show that no such tiling is possible. \Cref{question:main} is the main problem we pursue in this paper. In addressing \cref{question:main} here, we provide both positive and negative results. 

On the positive side, we show that many well-studied families of finite graded lattices are cores of upho lattices. Our first major result is the following.

\begin{thm} \label{thm:super_intro}
Any member of a uniform sequence of supersolvable geometric lattices is the core of some upho lattice.
\end{thm}

Supersolvable lattices were introduced by Stanley in~\cite{stanley1972supersolvable}, and uniform sequences of lattices were introduced by Dowling in~\cite{dowling1973class}. These two notions represent two different ways that a finite lattice can have a recursive structure. Examples of uniform sequences of supersolvable geometric lattices include:
\begin{itemize}
\item the finite Boolean lattices $B_n$, i.e., the lattices of subsets of $\{1,2,\ldots,n\}$;
\item the $q$-analogues $B_n(q)$ of $B_n$, i.e., the lattices of $\mathbb{F}_q$-subspaces of $\mathbb{F}_q^n$;
\item the partition lattices $\Pi_{n}$, i.e., the lattices of set partitions of $\{1,2,\ldots,n\}$;
\item the Type B partition lattices $\Pi^B_{n}$, i.e., the intersection lattices of the Type~$B_n$ Coxeter hyperplane arrangements;
\item (generalizing the previous two items) the Dowling lattices~\cite{dowling1973qanalog, dowling1973class} $Q_n(G)$ associated to any finite group $G$.
\end{itemize}
Hence, these are all cores of upho lattices. We discuss these examples in detail, providing explicit descriptions of the upho lattices for which they are cores. \Cref{fig:partitions} depicts one upho lattice produced via our construction.

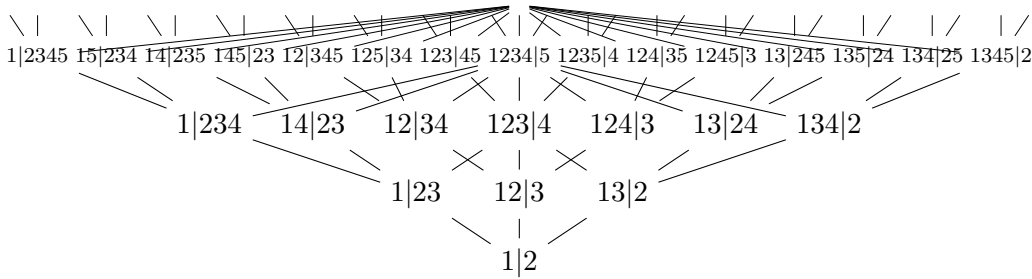
\begin{figure}

\begin{tikzpicture}[scale=0.915]
\node (A) at (0,0) {$1|2$};

\node (B) at (-1.5,1){$1|23$};
\node (C) at (0,1){$12|3$};
\node (D) at (1.5,1){$13|2$};

\node (E) at (-4.5,2){$1|234$};
\node (F) at (-3,2){$14|23$};
\node (G) at (-1.5,2){$12|34$};
\node (H) at (0,2){$123|4$};
\node (I) at (1.5,2){$124|3$};
\node (J) at (3,2){$13|24$};
\node (K) at (4.5,2){$134|2$};

\node (L) at (-7,3){\scriptsize $1|2345$};
\node (M) at (-6,3){\scriptsize $15|234$};
\node (N) at (-5,3){\scriptsize $14|235$};
\node (O) at (-4,3){\scriptsize $145|23$};
\node (P) at (-3,3){\scriptsize$12|345$};
\node (Q) at (-2,3){\scriptsize $125|34$};
\node (R) at (-1,3){\scriptsize $123|45$};
\node (S) at (0,3){\scriptsize $1234|5$};
\node (T) at (1,3){\scriptsize $1235|4$};
\node (U) at (2,3){\scriptsize $124|35$};
\node (V) at (3,3){\scriptsize $1245|3$};
\node (W) at (4,3){\scriptsize $13|245$};
\node (X) at (5,3){\scriptsize $135|24$};
\node (Y) at (6,3){\scriptsize $134|25$};
\node (Z) at (7,3){\scriptsize $1345|2$};

\node (1) at (-7.5,3.75){};
\node (2) at (-7,3.75){};
\node (3) at (-6.5,3.75){};
\node (4) at (-6,3.75){};
\node (5) at (-5.5,3.75){};
\node (6) at (-5,3.75){};
\node (7) at (-4.5,3.75){};
\node (8) at (-4,3.75){};
\node (9) at (-3.5,3.75){};
\node (10) at (-3,3.75){};
\node (11) at (-2.5,3.75){};
\node (12) at (-2,3.75){};
\node (13) at (-1.5,3.75){};
\node (14) at (-1,3.75){};
\node (15) at (-0.5,3.75){};
\node (16) at (0,3.75){};
\node (17) at (0.5,3.75){};
\node (18) at (1,3.75){};
\node (19) at (1.5,3.75){};
\node (20) at (2,3.75){};
\node (21) at (2.5,3.75){};
\node (22) at (3,3.75){};
\node (23) at (3.5,3.75){};
\node (24) at (4,3.75){};
\node (25) at (4.5,3.75){};
\node (26) at (5,3.75){};
\node (27) at (5.5,3.75){};
\node (28) at (6,3.75){};
\node (29) at (6.5,3.75){};
\node (30) at (7,3.75){};
\node (31) at (7.5,3.75){};

\draw (A) -- (B);
\draw (A) -- (C);
\draw (A) -- (D);

\draw (B) -- (E);
\draw (B) -- (F);
\draw (B) -- (H);
\draw (C) -- (G);
\draw (C) -- (H);
\draw (C) -- (I);
\draw (D) -- (H);
\draw (D) -- (J);
\draw (D) -- (K);

\draw (E) -- (L);
\draw (E) -- (M);
\draw (E) -- (S);
\draw (F) -- (N);
\draw (F) -- (O);
\draw (F) -- (S);
\draw (G) -- (P);
\draw (G) -- (Q);
\draw (G) -- (S);
\draw (H) -- (R);
\draw (H) -- (S);
\draw (H) -- (T);
\draw (I) -- (S);
\draw (I) -- (U);
\draw (I) -- (V);
\draw (J) -- (S);
\draw (J) -- (W);
\draw (J) -- (X);
\draw (K) -- (S);
\draw (K) -- (Y);
\draw (K) -- (Z);

\draw[ultra thin] (L) -- (1);
\draw[ultra thin] (L) -- (2);
\draw[ultra thin] (L) -- (16);
\draw[ultra thin] (M) -- (3);
\draw[ultra thin] (M) -- (4);
\draw[ultra thin] (M) -- (16);
\draw[ultra thin] (N) -- (5);
\draw[ultra thin] (N) -- (6);
\draw[ultra thin] (N) -- (16);
\draw[ultra thin] (O) -- (7);
\draw[ultra thin] (O) -- (8);
\draw[ultra thin] (O) -- (16);
\draw[ultra thin] (P) -- (9);
\draw[ultra thin] (P) -- (10);
\draw[ultra thin] (P) -- (16);
\draw[ultra thin] (Q) -- (11);
\draw[ultra thin] (Q) -- (12);
\draw[ultra thin] (Q) -- (16);
\draw[ultra thin] (R) -- (13);
\draw[ultra thin] (R) -- (14);
\draw[ultra thin] (R) -- (16);
\draw[ultra thin] (S) -- (15);
\draw[ultra thin] (S) -- (16);
\draw[ultra thin] (S) -- (17);
\draw[ultra thin] (T) -- (16);
\draw[ultra thin] (T) -- (18);
\draw[ultra thin] (T) -- (19);
\draw[ultra thin] (U) -- (16);
\draw[ultra thin] (U) -- (20);
\draw[ultra thin] (U) -- (21);
\draw[ultra thin] (V) -- (16);
\draw[ultra thin] (V) -- (22);
\draw[ultra thin] (V) -- (23);
\draw[ultra thin] (W) -- (16);
\draw[ultra thin] (W) -- (24);
\draw[ultra thin] (W) -- (25);
\draw[ultra thin] (X) -- (16);
\draw[ultra thin] (X) -- (26);
\draw[ultra thin] (X) -- (27);
\draw[ultra thin] (Y) -- (16);
\draw[ultra thin] (Y) -- (28);
\draw[ultra thin] (Y) -- (29);
\draw[ultra thin] (Z) -- (16);
\draw[ultra thin] (Z) -- (30);
\draw[ultra thin] (Z) -- (31);
\end{tikzpicture}
\caption{Partitions of sets of the form $\{1,2,\ldots,n\}$ into $2$ blocks, ordered by refinement. This is an upho lattice with core $\Pi_3$.} \label{fig:partitions}
\end{figure}

In addition to combinatorial constructions, we also explore algebraic constructions of upho lattices. Monoids provide one algebraic source of upho lattices, as the following lemma explains.

\begin{lemma}{(c.f.~\cite[Lemma~5.1]{gao2020upho})} \label{lem:monoid_intro}
Let $M$ be a finitely generated monoid whose defining relations are homogeneous. If $M$ is left cancellative and every pair of elements in $M$ have a least common right multiple, then $(M,\leq_L)$ is an upho lattice, where $\leq_L$ denotes the partial order of left divisibility.
\end{lemma}

A class of monoids satisfying the conditions of \cref{lem:monoid_intro} are the (homogeneous) Garside monoids~\cite{dehornoy1999gaussian, dehornoy2015foundations}. The core of a Garside monoid consists of its simple elements. Examples of lattices of simple elements in Garside monoids include:
\begin{itemize}
\item the weak order of a finite Coxeter group $W$;
\item the noncrossing partition lattice of a finite Coxeter group $W$.
\end{itemize}
Hence, these are also cores of upho lattices. \Cref{fig:monoid} depicts an upho lattice of this form. We review these examples coming from Garside and Coxeter theory in detail.

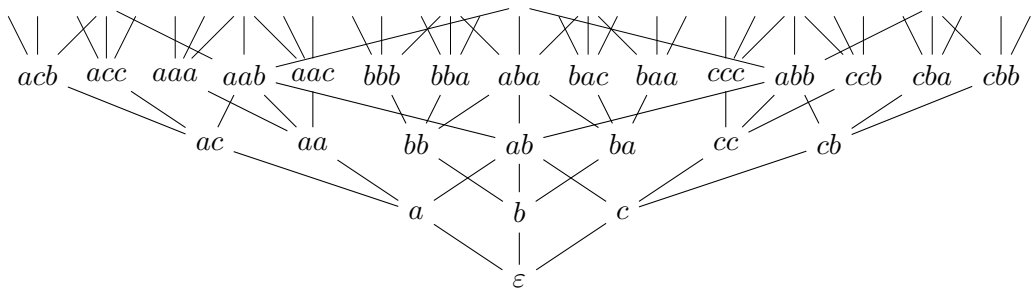
\begin{figure}
\begin{center}
\begin{tikzpicture}[scale=0.915]

\node (A) at (0,0) {$\varepsilon$};

\node (B) at (-1.5,1){$a$};
\node (C) at (0,1){$b$};
\node (D) at (1.5,1){$c$};

\node (E) at (-4.5,2){$ac$};
\node (F) at (-3,2){$aa$};
\node (G) at (-1.5,2){$bb$};
\node (H) at (0,2){$ab$};
\node (I) at (1.5,2){$ba$};
\node (J) at (3,2){$cc$};
\node (K) at (4.5,2){$cb$};

\node (L) at (-7,3){$acb$};
\node (M) at (-6,3){$acc$};
\node (N) at (-5,3){$aaa$};
\node (O) at (-4,3){$aab$};
\node (P) at (-3,3){$aac$};
\node (Q) at (-2,3){$bbb$};
\node (R) at (-1,3){$bba$};
\node (S) at (0,3){$aba$};
\node (T) at (1,3){$bac$};
\node (U) at (2,3){$baa$};
\node (V) at (3,3){$ccc$};
\node (W) at (4,3){$abb$};
\node (X) at (5,3){$ccb$};
\node (Y) at (6,3){$cba$};
\node (Z) at (7,3){$cbb$};

\node (1) at (-7.5,4){};
\node (2) at (-7,4){};
\node (3) at (-6.5,4){};
\node (4) at (-6,4){};
\node (5) at (-5.5,4){};
\node (6) at (-5,4){};
\node (7) at (-4.5,4){};
\node (8) at (-4,4){};
\node (9) at (-3.5,4){};
\node (10) at (-3,4){};
\node (11) at (-2.5,4){};
\node (12) at (-2,4){};
\node (13) at (-1.5,4){};
\node (14) at (-1,4){};
\node (15) at (-0.5,4){};
\node (16) at (0,4){};
\node (17) at (0.5,4){};
\node (18) at (1,4){};
\node (19) at (1.5,4){};
\node (20) at (2,4){};
\node (21) at (2.5,4){};
\node (22) at (3,4){};
\node (23) at (3.5,4){};
\node (24) at (4,4){};
\node (25) at (4.5,4){};
\node (26) at (5,4){};
\node (27) at (5.5,4){};
\node (28) at (6,4){};
\node (29) at (6.5,4){};
\node (30) at (7,4){};
\node (31) at (7.5,4){};

\draw (A) -- (B);
\draw (A) -- (C);
\draw (A) -- (D);

\draw (B) -- (E);
\draw (B) -- (F);
\draw (B) -- (H);
\draw (C) -- (G);
\draw (C) -- (H);
\draw (C) -- (I);
\draw (D) -- (H);
\draw (D) -- (J);
\draw (D) -- (K);

\draw (E) -- (L);
\draw (E) -- (M);
\draw (E) -- (O);
\draw (F) -- (N);
\draw (F) -- (O);
\draw (F) -- (P);
\draw (G) -- (Q);
\draw (G) -- (R);
\draw (G) -- (S);
\draw (H) -- (O);
\draw (H) -- (S);
\draw (H) -- (W);
\draw (I) -- (S);
\draw (I) -- (T);
\draw (I) -- (U);
\draw (J) -- (V);
\draw (J) -- (W);
\draw (J) -- (X);
\draw (K) -- (W);
\draw (K) -- (Y);
\draw (K) -- (Z);

\draw[ultra thin] (L) -- (1);
\draw[ultra thin] (L) -- (2);
\draw[ultra thin] (M) -- (3);
\draw[ultra thin] (M) -- (5);
\draw[ultra thin] (N) -- (6);
\draw[ultra thin] (N) -- (7);
\draw[ultra thin] (P) -- (9);
\draw[ultra thin] (P) -- (10);
\draw[ultra thin] (Q) -- (11);
\draw[ultra thin] (Q) -- (12);
\draw[ultra thin] (R) -- (13);
\draw[ultra thin] (R) -- (15);
\draw[ultra thin] (T) -- (17);
\draw[ultra thin] (T) -- (19);
\draw[ultra thin] (U) -- (20);
\draw[ultra thin] (U) -- (21);
\draw[ultra thin] (V) -- (22);
\draw[ultra thin] (V) -- (23);
\draw[ultra thin] (X) -- (25);
\draw[ultra thin] (X) -- (26);
\draw[ultra thin] (Y) -- (27);
\draw[ultra thin] (Y) -- (29);
\draw[ultra thin] (Z) -- (30);
\draw[ultra thin] (Z) -- (31);

\draw[ultra thin] (L) -- (4);
\draw[ultra thin] (M) -- (4);
\draw[ultra thin] (O) -- (4);
\draw[ultra thin] (N) -- (8);
\draw[ultra thin] (O) -- (8);
\draw[ultra thin] (P) -- (8);
\draw[ultra thin] (Q) -- (14);
\draw[ultra thin] (R) -- (14);
\draw[ultra thin] (S) -- (14);
\draw[ultra thin] (O) -- (16);
\draw[ultra thin] (S) -- (16);
\draw[ultra thin] (W) -- (16);
\draw[ultra thin] (S) -- (18);
\draw[ultra thin] (T) -- (18);
\draw[ultra thin] (U) -- (18);
\draw[ultra thin] (V) -- (24);
\draw[ultra thin] (W) -- (24);
\draw[ultra thin] (X) -- (24);
\draw[ultra thin] (W) -- (28);
\draw[ultra thin] (Y) -- (28);
\draw[ultra thin] (Z) -- (28);
\end{tikzpicture}
\end{center}
\caption{The dual braid monoid $\langle a,b,c\mid ab=bc=ca\rangle$ associated to the symmetric group $S_3$. This is an upho lattice with core the noncrossing partition lattice of $S_3$.} \label{fig:monoid}
\end{figure}

On the negative side, we show that there are various obstructions which prevent arbitrary finite graded lattices from being realized as cores of upho lattices. There are restrictions on the characteristic polynomial of the lattice coming from the equation~\eqref{eqn:rank_char}. There are also some structural obstructions, requiring the lattice to be partly self-similar. These obstructions allow us to show, for instance, that the following plausible candidates cannot in fact be realized as cores:
\begin{itemize}
\item the face lattice of the $n$-dimensional cross polytope, and the face lattice of the $n$-dimensional hypercube, for $n \geq 3$;
\item the bond lattice of the cycle graph $C_n$, for $n \geq 4$;
\item (generalizing the previous item) the lattice of flats of the uniform matroid $U(k,n)$, for $2 < k < n$.
\end{itemize}

The upshot is that \cref{question:main} is quite subtle: it can be difficult to recognize when a given finite graded lattice is the core of an upho lattice. Many well-behaved finite lattices are cores of upho lattices, but many too are not.

To conclude this introduction, let us also discuss future directions we are pursuing. A question naturally suggested by our work here is the following.

\begin{question} \label{question:counting}
For a finite graded lattice $L$, let $\kappa(L)$ denote the cardinality of the collection of upho lattices $\mathcal{L}$ with core $L$. How does this function $\kappa(L)$ behave?
\end{question}

Notice that \cref{question:counting} essentially asks how big the difference between answering \cref{question:main} and resolving \cref{prob:classify} is. In work in progress joint with Joel Lewis~\cite{hopkins2024upho2}, we address \cref{question:counting}. On the one hand, we will show that $\kappa(L)$ is finite if $L$ has no automorphisms, suggesting that possibly $\kappa(L)$ is finite for all finite lattices $L$. On the other hand, we will show that $\kappa(L)$ is unbounded even when we restrict to lattices $L$ of rank two.

Finally, if completely resolving \cref{prob:classify} is too difficult, we might instead hope to classify some subvarieties of upho lattices. Two of the most important subvarieties of lattices are the distributive lattices and the modular lattices. In planned future work, we will explore distributive and modular upho lattices. 

It is easy to see from the representation theorem for locally finite distributive lattices that the only upho distributive lattices are $\mathbb{N}^d$. Upho modular lattices are more interesting. Stanley observed that if $R$ is a discrete valuation ring with finite residue field, then the lattice of full rank submodules of the free module $R^d$ gives an upho modular lattice (see~\cite{stanley2019MO} and~\cite[Conjecture~1.1]{gao2020upho}). Slightly extending this observation, we can show any (sufficiently symmetric) affine building gives rise to an upho modular lattice. Conjecturally, all upho modular lattices are of this form.

The rest of the paper is structured as follows. In \cref{sec:prelim}, we go over some definitions and preliminary results. In \cref{sec:super}, we construct upho lattices from uniform sequences of supersolvable lattices. In \cref{sec:monoid}, we explain how monoids give rise to upho lattices. Finally, in \cref{sec:obstruct}, we discuss obstructions to realizing a finite graded lattice as the core of an upho lattice.

\subsection*{Acknowledgments}

I thank the following people for useful comments related to this work: Yibo Gao, Joel Lewis, Vic Reiner, David Speyer, Richard Stanley, Benjamin Steinberg, Nathan Williams, and Gjergji Zaimi. SageMath~\cite{Sage} was an important computational aid for this research. Finally, I thank the anonymous referees for their careful reading of the paper and their helpful comments.

\section{Preliminaries} \label{sec:prelim}

In this section, we review some basics regarding posets and upho posets. We generally stick to standard notation for posets, as laid out of instance in~\cite[\S 3]{stanley2012ec1}. We use $\mathbb{N}\coloneqq \{0,1,\ldots\}$ to denote the natural numbers, $\mathbb{Z}$ to denote the integers, $\mathbb{Q}$ the rationals, and $\mathbb{R}$ the reals.

\subsection{Poset basics}

Let $P=(P,\leq)$ be a poset. We use standard conventions like writing $y \geq x$ to mean $x \leq y$, writing $x < y$ to mean $x \leq y$ and $x \neq y$, and so on. We also routinely identify any subset $S \subseteq P$ with the corresponding induced subposet $S=(S,\leq)$.

\subsubsection{Basic terminology}

An \dfn{interval} of $P$ is a subset $[x,y] \coloneqq \{z \in P\colon x \leq z \leq y\}$ for $x \leq y\in P$. The poset $P$ is \dfn{locally finite} if every interval of $P$ is finite. 

For~$x, y \in P$, we say $x$ is \dfn{covered} by $y$, written $x \lessdot y$, if $x < y$ and there is no~$z \in P$ with $x < z < y$. If $P$ is locally finite, then the partial order $\leq$ is the reflexive, transitive closure of the cover relation $\lessdot$.

The \dfn{Hasse diagram} of~$P$ is the directed graph whose vertices are the elements of~$P$ with an edge from $x$ to $y$ when $x \lessdot y$. We draw the Hasse diagram of~$P$ in the plane, with $x$ below $y$ if there is an edge from~$x$ to $y$ (and therefore we do not draw the arrows on the edges). Since every locally finite poset is determined by its Hasse diagram, and the posets we study will be locally finite, we will represent posets by their Hasse diagrams.

A \dfn{chain} of $P$ is a totally ordered subset, i.e., a subset $C \subseteq P$ for which any two elements in $C$ are comparable. An \dfn{antichain} of $P$ is a subset $A \subseteq P$ for which any two elements in $A$ are incomparable. We say a chain is maximal if it is maximal by inclusion among chains, and similarly for antichains. 

An \dfn{order filter} of $P$ is an upwards-closed subset, i.e., a subset $F\subseteq P$ such that if $x \in F$ and $x \leq y$ then $y \in F$. Dually, an \dfn{order ideal} of $P$ is a downwards-closed subset, i.e., a subset $I\subseteq P$ such that if $y \in I$ and $x \leq y$ then $x\in I$. An order filter (respectively, order ideal) is principal if it is of the form $V_p\coloneqq \{q \in P\colon p \leq q\}$ (resp., of the form $\Lambda_p \coloneqq \{q \in P\colon q\leq p\}$) for some $p\in P$.

A \dfn{minimum} of $P$, which we always denote by $\hat{0}\in P$, is an element with $\hat{0}\leq x$ for all $x \in P$. Dually, a \dfn{maximum}, denoted $\hat{1}\in P$, is an element with $x \leq \hat{1}$ for all $x\in P$. Clearly, minima and maxima are unique if they exist. If $P$ has a minimum $\hat{0}$, then we call $s\in P$ an \dfn{atom} if $\hat{0} \lessdot s$. Dually, if $P$ has a maximum $\hat{1}$, then we call $t\in P$ a \dfn{coatom} if $t \lessdot \hat{1}$. 

\subsubsection{New posets from old} 

The \dfn{dual poset}  $P^{*}$ of $P$ is the poset with the same set of elements but with the opposite order, i.e., $x \leq_{P^*} y$ if and only if $y \leq_{P} x$.

Now let $Q$ be another poset. The \dfn{direct sum} $P+Q$ of $P$ and~$Q$ is the poset whose set of elements is the (disjoint) union $P\sqcup Q$, with $x \leq_{P + Q} y$ if either $x,y \in P$ and $x \leq_{P} y$, or $x,y\in Q$ and $x \leq_{Q} y$. For a positive integer $n\geq 1$, we denote the direct sum of $n$ copies of $P$ by $n\cdot P$. 

Meanwhile, the \dfn{ordinal sum} $P \oplus Q$ of $P$ and $Q$  is the poset whose elements are the (disjoint) union $P\sqcup Q$, with $x \leq_{P \oplus Q} y$ if either $x,y \in P$ and $x \leq_{P} y$, or $x,y\in Q$ and $x \leq_{Q} y$, or $x \in P$ and $y \in Q$. In other words, $P+Q$ is obtained by placing $P$ and $Q$ ``side-by-side,'' while $P\oplus Q$ is obtained by placing $P$ ``below'' $Q$.

The \dfn{direct product} $P\times Q$ of $P$ and $Q$ is the poset whose set of elements is the (Cartesian) product $P\times Q$, with $(p_1,q_1) \leq_{P\times Q} (p_2,q_2)$ if $p_1 \leq_{P} p_2$ and $q_1 \leq_{Q} q_2$. For a positive integer $n\geq 1$, we denote the direct product of $n$ copies of $P$ by $P^n$.

\subsubsection{M\"{o}bius functions}

Suppose for the moment that $P$ is locally finite, and let $\mathrm{Int}(P)$ denote the set of intervals of $P$. The \dfn{M\"{o}bius function} $\mu\colon \mathrm{Int}(P)\to \mathbb{Z}$ of $P$ is defined recursively by
\[ \mu(x,x) = 1 \textrm{ for all $x \in P$}; \qquad \mu(x,y) = -\sum_{x \leq z < y} \mu(x,z) \textrm{ for all $x<y \in P$},\]
where we use the standard notational shorthand $\mu(x,y) = \mu([x,y])$. The most important application of M\"{o}bius functions is the M\"{o}bius inversion formula (see~\cite[\S 3.7]{stanley2012ec1}), a kind of generalization of the principle of inclusion-exclusion to any poset.

The M\"{o}bius function of a product of posets decomposes as a product of M\"{o}bius functions. In other words, we have $\mu_{P\times Q}( (p_1,q_1), (p_2,q_2)) = \mu_P(p_1,p_2) \cdot \mu_Q(q_1,q_2)$ for all $(p_1,q_1) \leq (p_2,q_2) \in P\times Q$ (see \cite[Proposition~3.8.2]{stanley2012ec1}).

\subsubsection{Lattices}

For $x,y \in P$, an \dfn{upper bound} of $x$ and $y$ is a $z\in P$ with $x \leq z$ and $y \leq z$, and the \dfn{join} (or \dfn{least upper bound}) of $x$ and $y$, denoted $x \vee y$, is the minimum among all upper bounds of $x$ and $y$, if such a minimum exists. Dually, a \dfn{lower bound} of $x$ and $y$ is a $z\in P$ with $z \leq x$ and $z \leq y$, and the \dfn{meet} (or \dfn{greatest lower bound}) of $x$ and $y$, denoted $x \wedge y$, is the maximum among all lower bounds of $x$ and $y$, if such a maximum exists. If $x \vee y$ exists for every $x,y \in P$, then $P$ is a \dfn{join semilattice}. Dually, if~$x \wedge y$ exists for every $x,y \in P$, then $P$ is a \dfn{meet semilattice}. The poset $P$ is a \dfn{lattice} if it is both a join and meet semilattice. 

Now let $L$ be a lattice. The operations of $\vee$ and $\wedge$ are associative and commutative, and therefore for any finite, nonempty subset $S = \{x_1,\ldots,x_n\}\subseteq L$ we can set~$\bigvee S \coloneqq x_1 \vee \cdots \vee x_n $ and $\bigwedge S \coloneqq x_1 \wedge \cdots \wedge x_n$. If $L$ has a minimum $\hat{0}$ then by convention we set~$\bigvee \emptyset \coloneqq \hat{0}$, and dually if $L$ has a maximum $\hat{1}$ we set~$\bigwedge \emptyset \coloneqq \hat{1}$. A \emph{finite} lattice $L$ always has a minimum $\hat{0}=\bigwedge L$ and a maximum $\hat{1}=\bigvee L$.

If $L$ has a minimum $\hat{0}$ and a maximum $\hat{1}$, then a \dfn{complement} of an element $x \in L$ is a $y\in L$ with $x \wedge y = \hat{0}$ and $x \vee y = \hat{1}$.

The lattice $L$ is \dfn{distributive} if the operation of meet distributes over that of join, i.e., $x \wedge (y \vee z) = (x \wedge y) \vee (x \wedge z)$ for all $x,y,z \in L$. It is well-known that this is also equivalent to join distributing over meet, i.e., $x \vee (y \wedge z) = (x \vee y) \wedge (x \vee z)$ for all~$x,y,z \in L$. 

\begin{example}
For any poset $P$, we use $J(P)$ to denote the poset of order ideals of~$P$, ordered by inclusion. Then, $J(P)$ is always a distributive lattice, where the operations of join and meet are union and intersection, respectively. As a variant of this construction, we use $J_{\mathrm{fin}}(P)$ to denote the finite order ideals of $P$, which again always gives a distributive lattice. We note that $J(P+Q) = J(P)\times J(Q)$ and similarly $J_{\mathrm{fin}}(P+Q) = J_{\mathrm{fin}}(P)\times J_{\mathrm{fin}}(Q)$.
\end{example}

The lattice $L$ is \dfn{modular} if whenever $x\leq y \in L$, we have $x \vee (z \wedge y) = (x \vee z) \wedge y$ for all $z\in L$. Observe that modularity is a weaker condition than distributivity: all distributive lattices are modular, but most modular lattices are not distributive. 

We note that the product $L_1 \times L_2$ of two lattices $L_1$ and $L_2$ remains a lattice, and similarly if we append the adjectives ``distributive'' or ``modular.'' Also, any interval in a lattice is a lattice, and similarly if we append the adjectives ``distributive'' or ``modular.'' These properties follow from the fact that lattices, distributive lattices, and modular lattices, are varieties in the sense of universal algebra.

\subsubsection{Convention for finite versus infinite posets}

We will routinely work with both finite and infinite posets, although the posets will always be at least locally finite. For clarity, we now use the following convention: normal script letters (like $P$ or $L$) denote finite posets, while caligraphic letters (like $\mathcal{P}$ or~$\mathcal{L}$) denote infinite posets.

\subsection{Finite graded posets}

Let $P$ be a finite poset. For a nonnegative integer $n \geq 0$, we say that $P$ is \dfn{$n$-graded} if $P$ has a minimum $\hat{0}$, a maximum $\hat{1}$, and we can write $P=\bigsqcup_{i=0}^{n} P_i$ as a disjoint union such that every maximal chain of $P$ is of the form $\hat{0} = x_0 \lessdot x_1 \lessdot \cdots \lessdot x_n = \hat{1}$ with $x_i \in P_i$. In this case, the \dfn{rank function}~$\rho \colon P \to \mathbb{N}$ is defined by setting $\rho(x) \coloneqq i$ for $x \in P_i$. Equivalently, $\rho(\hat{0}) = 0$ and $\rho(y) = \rho(x)+1$ whenever $x \lessdot y \in P$.

\begin{example}
For any positive integer $n\geq 1$, we let $[n]\coloneqq \{1,2,\ldots,n\}$. We view~$[n]$ as a poset, with the usual total order. This chain poset~$[n]$ is the most basic example of a finite $(n-1)$-graded poset.
\end{example}

We say that the finite poset $P$ is graded if it is $n$-graded for some $n$. In this case, we say that the rank of $P$ is $n$ and, slightly abusing notation, write $\rho(P) \coloneqq n$. If~$P$ and $Q$ are two finite graded posets, then their product $P\times Q$ is also graded of rank $\rho(P\times Q) = \rho(P) + \rho(Q)$. Also, any interval in a finite graded poset is graded.

\subsubsection{Generating polynomials for finite graded posets}

Now assume that $P$ is graded. The \dfn{rank generating polynomial} of $P$ is 
\[F(P; x) \coloneqq \sum_{p \in P} x^{\rho(p)}.\]
The \dfn{characteristic polynomial} of $P$ is 
\[ \chi(P; x) \coloneqq \sum_{p \in P} \mu(\hat{0},p) \, x^{\rho(P)-\rho(p)}.\]
The exponent of $x$ in each term of the characteristic polynomial $\chi(P; x)$ records the corank $\rho(P)-\rho(p)$ of the element $p \in P$. Using the corank in the characteristic polynomial is very standard, but, for reasons that will become clear soon, we need a version of the characteristic polynomial where the exponent records the usual rank instead. Hence, we define the \dfn{reciprocal characteristic polynomial} of $P$ to be
\[ \chi^*(P; x) \coloneqq \sum_{p \in P} \mu(\hat{0},p) \, x^{\rho(p)}.\]
Observe that $\chi^*(P; x) = x^{\rho(P)} \cdot \chi(P;x^{-1})$.

These invariants of finite graded posets all play nicely with products. Namely, $F(P\times Q; x) = F(P; x) \cdot F(Q; x)$ and $\chi^*(P\times Q; x) = \chi^*(P; x) \cdot \chi^*(Q; x)$.

\subsubsection{Finite graded lattices}

Here we are most interested in finite graded lattices.

\begin{example} \label{ex:boolean}
The rank $n$ \dfn{Boolean lattice} $B_n$ is the poset of subsets of~$[n]$, ordered by inclusion. $B_n$ is a finite graded lattice, with $\rho(S) = \#S$ for all $S\in B_n$. Its M\"{o}bius function is given by $\mu(S,T) = (-1)^{\#T\setminus S}$ for all $S\leq T \in B_n$. Hence, $F(B_n; x) = \sum_{k=0}^{n}\binom{n}{k}x^k= (1+x)^n$ and $\chi^*(B_n; x) = \sum_{k=0}^{n} (-1)^k\binom{n}{k} x^k= (1-x)^n$. These formulas can also be seen from the fact that $B_n=J(n \cdot [1]) = [2]^n$.
\end{example}

\begin{example} \label{ex:distributive}
Birkhoff's representation theorem for finite distributive lattices says that every finite distributive lattice $L$ has the form $L = J(P)$ for a unique finite poset~$P$, the subposet of join-irreducible elements of $L$; see \cite[\S 3.3]{birkhoff1967lattice} or~\cite[\S3.4]{stanley2012ec1}. So let~$L=J(P)$ be a finite distributive lattice. Then $L$ is graded, with $\rho(I) = \#I$ for~$I \in J(P)$. Its M\"{o}bius function is given by 
\[\mu(I,I') = \begin{cases} (-1)^{\#I' \setminus I} &\textrm{if $I'\setminus I$ is an antichain of $P$}; \\
0 &\textrm{otherwise},\end{cases}\]
for $I\leq I' \in J(P)$ (see~\cite[Example 3.9.6]{stanley2012ec1}). Hence,  $F(L; x) = \sum_{I \in J(P)} x^{\#I}$ and $\chi^*(L; x) = \sum_{I \subseteq \min(P)}(-x)^{\#I} = (1-x)^{\#\min(P)}$, where $\min(P)$ is the set of minimal elements of $P$. Observe how this example generalizes \cref{ex:boolean}.
\end{example}

\Cref{ex:distributive} explains that all finite distributive lattices are graded. More generally, all finite modular lattices are graded. In fact, a finite lattice $L$ is modular if and only if $L$ is graded and $\rho(p) + \rho(q) = \rho(p \vee q) + \rho(p \wedge q)$ for all $p, q\in L$ (see, e.g., \cite[\S 2.8]{birkhoff1967lattice} or~\cite[\S3.3]{stanley2012ec1}). 

\begin{example} \label{ex:q-boolean}
Let $q$ be a prime power, and $\mathbb{F}_q$ the finite field with $q$ elements. We denote by $B_n(q)$ the poset of $\mathbb{F}_q$-subspaces of the vector space $\mathbb{F}_q^n$, ordered by inclusion. This \dfn{subspace lattice} $B_n(q)$, also known as the $q$-analogue of~$B_n$, is a finite modular lattice. Its rank function is $\rho(U) = \dim(U)$ for $U \in B_n(q)$, and its M\"{o}bius function is $\mu(U,V) = (-1)^{k}q^{\binom{k}{2}}$, where $k=\dim(V)-\dim(U)$, for $U \leq V \in B_n(q)$ (see \cite[Example~3.10.2]{stanley2012ec1}). Hence, $F(B_n(q); x) = \sum_{k=0}^{n}\qbinom{n}{k}_q \, x^k$ and $\chi^*(B_n(q); x) = \sum_{k=0}^{n}(-1)^k q^{\binom{k}{2}} \qbinom{n}{k}_q \, x^k = (1-x)(1-qx)(1-q^2x)\cdots (1-q^{n-1}x)$. Here we used the standard notation for the $q$-binomial coefficient $\qbinom{n}{k}_q \coloneqq \frac{[n]_q!}{[k]_q![n-k]_q!}$, where $[n]_q \coloneqq \frac{(1-q^n)}{(1-q)} = 1+q +q^2+\cdots+q^{n-1}$ and $[n]_q! \coloneqq [n]_q \cdot [n-1]_q \cdots [1]_q$.
\end{example}

There are variants of the modular property for finite lattices that are very interesting from a combinatorial point of view. A finite lattice $L$ is \dfn{(upper) semimodular} if it is graded and satisfies $\rho(p) + \rho(q) \geq \rho(p \vee q) + \rho(p \wedge q)$ for all~$p, q\in L$. The lattice~$L$ is \dfn{atomic} if every element is a join of atoms. Finally, $L$ is \dfn{geometric} if it is both semimodular and atomic. For example, the modular lattices $B_n$ and $B_n(q)$ are geometric, since they are atomic. However, most geometric lattices are not modular. Geometric lattices are intensely studied, because they are precisely the lattices of flats of matroids (see~\cite[\S3.3]{stanley2012ec1}). Any interval in a semimodular lattice is semimodular, and, less obviously, any interval in a geometric lattice is geometric (see, e.g., \cite[\S8.9, Exercise 7]{birkhoff1967lattice} and \cite[Proposition~3.3.3]{stanley2012ec1}).

\begin{example} \label{ex:partition}
Recall that a partition of a set $X$ is a collection $\pi=\{B_1,\ldots,B_k\}$ of nonempty subsets of $X$ ($\emptyset \subsetneq B_1,\ldots,B_k \subseteq X$) that are pairwise disjoint ($B_i \cap B_j = \emptyset$ for $i \neq j$) and whose union is $X$ ($\cup_{i=1}^{k}B_i = X$). The sets $B_i \in \pi$ are called the blocks of $\pi$. The \dfn{partition lattice} $\Pi_n$ is the poset of partitions of $[n]$, ordered by refinement. In other words, for two partitions~$\pi$ and $\pi'$ of $[n]$, we have $\pi \leq \pi'$ if every block $B \in \pi$ satisfies $B\subseteq B'$ for some block $B'\in \pi'$. The partition lattice is a geometric lattice of rank $\rho(\Pi_n)=n-1$, with $\rho(\pi) = n-\#\pi$ for $\pi \in \Pi_n$. Hence, $F(\Pi_n; x) = \sum_{k=0}^{n}S(n,n-k) x^k$, where $S(n,k)$ are the Stirling number of the second kind. Furthermore, it is well-known that $\chi^*(\Pi_n; x)= \sum_{k=0}^{n}s(n,n-k) x^k = (1-x)(1-2x)\cdots (1-(n-1)x)$, where $s(n,k)$ are the (signed) Stirling number of the first kind (see \cite[Example 3.10.4]{stanley2012ec1}). 
\end{example}

\subsection{Infinite graded posets}

Let $\mathcal{P}$ be an infinite poset. We say $\mathcal{P}$ is \dfn{$\mathbb{N}$-graded} if~$\mathcal{P}$ has a minimum $\hat{0}$ and we can write $\mathcal{P}=\bigsqcup_{i=0}^{\infty}P_i$ as a disjoint union such that every maximal chain of $P$ is of the form $\hat{0} = x_0 \lessdot x_1 \lessdot x_2 \lessdot \cdots$ with $x_i \in P_i$. In this case, the \dfn{rank function} $\rho \colon \mathcal{P} \to \mathbb{N}$ is defined by setting $\rho(x) \coloneqq i$ for $x \in P_i$. Equivalently, $\rho(\hat{0}) = 0$ and~$\rho(y) = \rho(x)+1$ whenever~$x \lessdot y \in \mathcal{P}$.

 If $\mathcal{P}$ and $\mathcal{Q}$ are $\mathbb{N}$-graded, then $\mathcal{P}\times \mathcal{Q}$ is $\mathbb{N}$-graded. Also, any interval in a locally finite $\mathbb{N}$-graded poset is a finite graded poset.

\subsubsection{Generating functions for infinite graded posets}

Let $\mathcal{P}$ be an $\mathbb{N}$-graded poset. In order to define sensible analogues of the rank generating and characteristic polynomials for $\mathcal{P}$, we need to make a further finiteness assumption. So let us say that $\mathcal{P}$ is \dfn{finite type} if $\{p\in \mathcal{P}\colon \rho(p)=i\}$ is finite for each $i \in \mathbb{N}$. Observe that a finite type $\mathbb{N}$-graded poset is locally finite.

So now assume that $\mathcal{P}$ is a finite type $\mathbb{N}$-graded poset. Then we define the \dfn{rank generating function} of $\mathcal{P}$ to be 
\[F(\mathcal{P}; x) = \sum_{p \in \mathcal{P}} x^{\rho(p)}, \]
a formal power series in the variable $x$. And we define the \dfn{characteristic generating function} of $\mathcal{P}$ to be
\[\chi^*(\mathcal{P}; x) =\sum_{p\in \mathcal{P}} \mu(\hat{0},p) \, x^{\rho(p)}, \]
again, a formal power series. We write $\chi^*(\mathcal{P}; x)$ with an asterisk to emphasize that the characteristic generating function of an infinite poset $\mathcal{P}$ uses rank in the exponent, like the \emph{reciprocal} characteristic polynomial $\chi^*(P; x)$ of a finite poset $P$.

Again, these invariants play nicely with products: $F(\mathcal{P}\times\mathcal{Q}; x)=F(\mathcal{P}; x) \cdot F(\mathcal{Q}; x)$ and $\chi^*(\mathcal{P\times Q}; x) = \chi^*(\mathcal{P}; x) \cdot \chi^*(\mathcal{Q}; x)$.

\begin{example}
The set of natural numbers $\mathbb{N} = \{0,1,\ldots\}$, with their usual total order, is the most basic example of an $\mathbb{N}$-graded poset. In fact, $\mathbb{N}$ is a finite type $\mathbb{N}$-graded lattice with $F(\mathbb{N}; x) = \sum_{k=0}^{\infty}x^k = \frac{1}{1-x}$ and $\chi^*(\mathbb{N}; x) = 1-x$. Hence, for any positive integer $n \geq 1$, $\mathbb{N}^n$ is a finite type $\mathbb{N}$-graded lattice with $F(\mathbb{N}^n; x) = \frac{1}{(1-x)^n}$ and $\chi^*(\mathbb{N}^n; x) = (1-x)^n$.
\end{example}

\subsection{Upho posets}

We say that poset $\mathcal{P}$ is \dfn{upper homogeneous}, or ``\dfn{upho},'' if for every $p \in \mathcal{P}$, the principal order filter $V_p=\{q\in\mathcal{P}\colon q\geq p\}$ is isomorphic to~$\mathcal{P}$. To avoid trivialities, let us also require that $\mathcal{P}$ has at least two elements; then, it must be infinite.

\begin{example}
The natural numbers $\mathbb{N}$ form an upho poset. Similarly, the nonnegative rational numbers $\mathbb{Q}_{\geq 0}$, the nonnegative real number $\mathbb{R}_{\geq 0}$, and indeed the nonnegative numbers in any ordered field, form upho posets.
\end{example}

\begin{example}
Let $X$ be any infinite set. Then the poset of finite subsets of $X$, ordered by inclusion, is upho.
\end{example}

In this paper we are primarily concerned with upho posets (in fact, upho lattices). In order to be able to apply the tools of enumerative and algebraic combinatorics to study these posets, we must impose some finiteness conditions on them. Hence, from now on, {\bf all upho posets are assumed finite type $\mathbb{N}$-graded} unless otherwise specified. Of the preceding examples, only $\mathbb{N}$ is finite type $\mathbb{N}$-graded.

The product $\mathcal{P}\times \mathcal{Q}$ of two upho posets $\mathcal{P}$ and $\mathcal{Q}$ remains upho. So, for instance, $\mathbb{N}^n$ is an upho lattice for any $n \geq 1$. We will soon see many more examples of upho lattices, but for now $\mathbb{N}^n$ is a good prototypical example to have in mind.

\begin{remark}
Upho posets were introduced by Stanley~\cite{stanley2020upho, stanley2024theorems}. Stanley was mainly interested in planar upho posets (i.e., those with planar Hasse diagrams). In particular, for various planar upho posets~$\mathcal{P}$, he was interested in counting the maximal chains in~$[\hat{0},p]$ for $p \in \mathcal{P}$. When $\mathcal{P}=\mathbb{N}^2$, these numbers form Pascal's triangle. Thus, Stanley used these chain counts for other planar upho posets to produce analogues of Pascal's triangle~\cite{stanley2020some, stanley2024theorems}. All planar upho posets are meet semilattices, but most are not lattices. In fact, it is not hard to see that $\mathbb{N}$ and $\mathbb{N}^2$ are the only planar upho lattices. Planar upho posets have a rather simple structure, as described in~\cite{gao2020upho}. We will see that upho lattices can have a very intricate structure.
\end{remark}

\subsubsection{Rank and characteristic generating functions of upho posets}

The following important result on rank generating functions of upho posets can be proved by a simple application of M\"{o}bius inversion.

\begin{thm}[{\cite[Theorem 1]{hopkins2022note}}] \label{thm:upho_rgf}
For any upho poset $\mathcal{P}$,  $F(\mathcal{P}; x) = \chi^*(\mathcal{P}; x)^{-1}$.
\end{thm}

\subsubsection{Upho lattices and their cores}

Now suppose that $\mathcal{L}$ is an upho lattice. Then we define the \dfn{core} of $\mathcal{L}$ to be $L \coloneqq [\hat{0},s_1\vee s_2 \vee \cdots \vee s_r]\subseteq \mathcal{L}$, where $s_1,\ldots,s_r$ are the atoms of $\mathcal{L}$. Evidently, the core of an upho lattice is a finite graded lattice. The point of the core is the following corollary, which can be proved for instance using Rota's cross-cut theorem (see~\cite[Corollary 3.9.4]{stanley2012ec1}).

\begin{cor}[{\cite[Corollary 6]{hopkins2022note}}] \label{cor:upho_lat_rgf}
Let $\mathcal{L}$ be an upho lattice with core $L$. Then $\chi^*(\mathcal{L}; x) = \chi^*(L; x)$. Hence, from \cref{thm:upho_rgf}, we conclude $F(\mathcal{L}; x) =  \chi^*(L; x)^{-1}$.
\end{cor}

Note that \cref{cor:upho_lat_rgf} was stated as equation~\eqref{eqn:rank_char} in the introduction.

\begin{example}
For any $n \geq 1$, $\mathbb{N}^n$ is an upho lattice with core $B_n$, and indeed $F(\mathbb{N}^n; x) = \frac{1}{(1-x)^n} = \chi^*(B_n; x)^{-1}$.
\end{example}

With all the terminology and preliminary results fully explained, we now return to \cref{question:main}: which finite graded lattices $L$ arise as cores of upho lattices~$\mathcal{L}$? For example, we just saw that the Boolean lattices $B_n$ do. We will explore this question in the next three sections. 

\section{Upho lattices from sequences of finite lattices} \label{sec:super}

In this section, we will develop a method for producing upho lattices from limits of sequences of finite graded lattices which are appropriately embedded in one another. In order to make ``appropriately embedded in one another'' precise, we will need two notions from the theory of finite lattices: supersolvability, as introduced by Stanley in~\cite{stanley1972supersolvable}, and uniformity, as introduced by Dowling in~\cite{dowling1973class}. 

\subsection{Supersolvable semimodular lattices} \label{subsec:super}

Let $L$ be a (finite) semimodular lattice. An element $p\in L$ is \dfn{modular} if $\rho(p) + \rho(q) = \rho(p \vee q) + \rho(p \wedge q)$ for all~$q \in L$. For example, if $L$ is modular, then every element is modular. An important property of modularity is that it is transitive on ``lower'' intervals:

\begin{prop}[{See~\cite{stanley1971modular} or~\cite[Proposition 4.10(b)]{stanley2007hyperplane}}] \label{prop:mod_trans}
Let $L$ be a semimodular lattice. If $x$ is modular in $L$ and $y$ is modular in $[\hat{0},x]$, then $y$ is modular in~$L$.
\end{prop}

The semimodular lattice $L$ is \dfn{supersolvable} if it possesses a maximal chain $\hat{0} = x_0 \lessdot x_1 \lessdot \cdots \lessdot x_n = \hat{1}$ of modular elements. As suggested by \cref{prop:mod_trans}, we can build up a maximal chain of modular elements from the top down, coatom-by-coatom. For this reason, supersolvable lattices have a recursive structure. Moreover, as shown by Stanley~\cite{stanley1972supersolvable}, they enjoy many remarkable enumerative properties, the most prominent being that their characteristic polynomials factor.

\begin{thm}[{Stanley~\cite{stanley1971modular}~\cite[Theorem 4.1]{stanley1972supersolvable}}] \label{thm:super_stanley}
Let $L$ be a supersolvable semimodular lattice with  maximal chain of modular elements $x_0 \lessdot x_1 \lessdot \cdots \lessdot x_n$. For $i=1,\ldots,n$, set 
\[a_i \coloneqq \#\{\textrm{atoms } s\in L\colon s \leq x_i, s \not \leq x_{i-1}\}.\]
Then $\chi^*(L;x) = (1-a_1x) (1-a_2x) \cdots (1-a_n x)$.
\end{thm}

\begin{example}
The partition lattice $\Pi_n$ is a supersolvable geometric lattice. Indeed, $\pi_0 \lessdot \pi_1 \lessdot \cdots \lessdot \pi_{n-1} \in \Pi_n$ is a maximal chain of modular elements, where
\[\pi_i \coloneqq \{ \{1,2,\ldots,i+1\},\{i+2\}, \{i+3\},\cdots,\{n\} \}\]
for $i=0,1,\ldots,n-1$. Here $a_i=i$ for $i=1,\ldots,n-1$, so \cref{thm:super_stanley} gives that $\chi^*(\Pi_n;x) = (1-x)(1-2x)\cdots(1-(n-1)x)$, in agreement with what we said earlier.
\end{example}

\begin{remark}
In~\cite{stanley1972supersolvable}, Stanley defined the notion of supersolvability more generally for any finite lattice, not necessarily semimodular. It is possible that the results in this section could be extended to this broader class of (not necessarily semimodular) supersolvable lattices. However, the examples that we know all involve geometric lattices, so it is unclear, at the moment, what this greater generality would buy us.
\end{remark}

Next, we describe a procedure for ``trimming'' a supersolvable semimodular lattice to produce another one. Although this procedure is quite natural, we have not seen it anywhere in the literature. We use this trimming procedure later to guarantee that a certain limit poset is finite type $\mathbb{N}$-graded.

So let $L$ be a supersolvable semimodular lattice with a fixed maximal chain of modular elements $x_0 \lessdot x_1 \lessdot \cdots \lessdot x_n$. Relative to this chain, we define $\nu \colon L \to \mathbb{N}$ by $\nu(x) \coloneqq \min \{i \colon x \leq x_i\}$ for all $x\in L$. And then for a positive integer $k \geq 1$, we define
\[ L^{(k)} \coloneqq \{x \in L\colon \nu(x) - \rho(x) < k\}.\]
For instance, $L^{(1)} = \{x_0,x_1,\ldots,x_n\}$ and $L^{(n)} = L$. Note that $L^{(i)} \subseteq L^{(j)}$ for $i \leq j$.

\begin{lemma} \label{lem:super_trim}
For any $k \geq 1$, the subposet $L^{(k)}$ of $L$ is again a supersolvable semimodular lattice, with maximal chain of modular elements $x_0 \lessdot x_1 \lessdot \cdots \lessdot x_n$.
\end{lemma}

The rank of an element in $L^{(k)}$ is the same as its rank in $L$. And the join of two elements in $L^{(k)}$ is the same as their join in $L$. But the meet of two elements in $L^{(k)}$ may be different than their meet in $L$, so $L^{(k)}$ is not in general a sublattice of $L$.

\begin{example}
Again consider the partition lattice $\Pi_n$, with maximal chain of modular elements $\pi_0 \lessdot \pi_1 \lessdot \cdots \lessdot \pi_{n-1}$ as described above. Then, for any $k \geq 1$, 
\[\Pi^{(k)}_n = \{ \pi \in \Pi_n \colon \max\{i\colon \textrm{$i$ is in a non-singleton block of $\pi$}\}< n+k-\#\pi\}.\]
(Recall that a block is called a singleton if it has one element. Also, above we use the convention $\max(\emptyset) \coloneqq 0$.) For instance, \cref{fig:partitions_finite} depicts $\Pi^{(2)}_4$. In this figure we use the shorthand of writing a partition $\pi = \{B_1,\ldots,B_k\}$ as~$B_1 \mid B_2 \mid \cdots \mid B_k$. The perceptive reader may compare this figure to \cref{fig:partitions}.
\end{example}

\begin{figure}
\begin{tikzpicture}[scale=0.825]
\node (A) at (0,0) {$1|2|3|4$};

\node (B) at (-1.5,1){$1|23|4$};
\node (C) at (0,1){$12|3|4$};
\node (D) at (1.5,1){$13|2|4$};

\node (E) at (-4.5,2){$1|234$};
\node (F) at (-3,2){$14|23$};
\node (G) at (-1.5,2){$12|34$};
\node (H) at (0,2){$123|4$};
\node (I) at (1.5,2){$124|3$};
\node (J) at (3,2){$13|24$};
\node (K) at (4.5,2){$134|2$};

\node (S) at (0,3){$1234$};

\node (1) at (-7.5,3.75){};
\node (2) at (-7,3.75){};
\node (3) at (-6.5,3.75){};
\node (4) at (-6,3.75){};
\node (5) at (-5.5,3.75){};
\node (6) at (-5,3.75){};
\node (7) at (-4.5,3.75){};
\node (8) at (-4,3.75){};
\node (9) at (-3.5,3.75){};
\node (10) at (-3,3.75){};
\node (11) at (-2.5,3.75){};
\node (12) at (-2,3.75){};
\node (13) at (-1.5,3.75){};
\node (14) at (-1,3.75){};
\node (15) at (-0.5,3.75){};
\node (16) at (0,3.75){};
\node (17) at (0.5,3.75){};
\node (18) at (1,3.75){};
\node (19) at (1.5,3.75){};
\node (20) at (2,3.75){};
\node (21) at (2.5,3.75){};
\node (22) at (3,3.75){};
\node (23) at (3.5,3.75){};
\node (24) at (4,3.75){};
\node (25) at (4.5,3.75){};
\node (26) at (5,3.75){};
\node (27) at (5.5,3.75){};
\node (28) at (6,3.75){};
\node (29) at (6.5,3.75){};
\node (30) at (7,3.75){};
\node (31) at (7.5,3.75){};

\draw (A) -- (B);
\draw (A) -- (C);
\draw (A) -- (D);

\draw (B) -- (E);
\draw (B) -- (F);
\draw (B) -- (H);
\draw (C) -- (G);
\draw (C) -- (H);
\draw (C) -- (I);
\draw (D) -- (H);
\draw (D) -- (J);
\draw (D) -- (K);

\draw (E) -- (S);
\draw (F) -- (S);
\draw (G) -- (S);
\draw (H) -- (S);
\draw (I) -- (S);
\draw (J) -- (S);
\draw (K) -- (S);
\end{tikzpicture}
\caption{The ``trimmed'' partition lattice $\Pi^{(2)}_4$.} \label{fig:partitions_finite}
\end{figure}
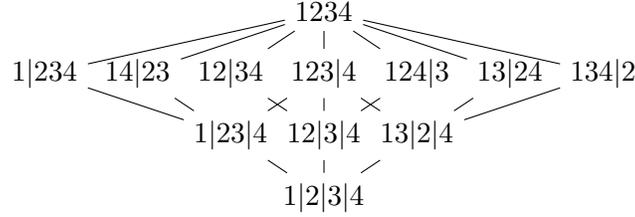

\begin{proof}[Proof of \cref{lem:super_trim}]
We start with two important claims.

{\bf Claim 1}: For $x, y \in L^{(k)}$, the join $x \vee y$, taken in $L$, is in $L^{(k)}$.

{\bf Claim 2}: For $x \in L^{(k)}$ and $i=0,1,\ldots,n$, the meet $x \wedge x_i$, taken in $L$, is in $L^{(k)}$.

Proof of Claim 1: Let $x,y \in L^{(k)}$. Set $c\coloneqq \max(\nu(x),\nu(y))$, $d\coloneqq \max(\rho(x),\rho(y))$. Since $\nu(x) - \rho(x) < k$ and $\nu(y) - \rho(y) < k$, we have $c-d < k$. And since $x_c$ is an upper bound for $x$ and $y$, we have $\nu(x\vee y) \leq c$. Similarly, since $x\vee y$ is greater than or equal to $x$ and to $y$, we have $\rho(x\vee y) \geq d$. Thus, $\nu(x\vee y)-\rho(x \vee y) \leq c - d < k$. 

Proof of Claim 2: Let $x \in L^{(k)}$ and $i \in \{0,1,\ldots,n-1\}$. If $x \leq x_i$, then $x \wedge x_i = x$ and the claim is clear. So suppose that $x \not \leq  x_i$. Since $x_{\nu(x)}$ is an upper bound for $x$ and $x_i$, we have $\rho(x \vee x_i) \leq \nu(x)$. And since $x_i$ is modular, we have $\rho(x) + \rho(x_i) = \rho(x \vee x_i) + \rho(x \wedge x_i)$, i.e., $\rho(x \wedge x_i) = \rho(x) + \rho(x_i) - \rho(x \vee x_i) \geq \rho(x) + i - \nu(x)$. Meanwhile, since $x_i$ is an upper bound for $x \wedge x_i$, we have $\nu(x \wedge x_i) \leq i$. Thus, $\nu(x \wedge x_i) - \rho(x \wedge x_i) \leq \nu(x) - \rho(x) < k$.

We proceed to prove the lemma. First, let us explain why $L^{(k)}$ is a lattice. Notice Claim 1 already implies that $L^{(k)}$ is a sub-join semilattice of $L$. Then, since $L^{(k)}$ is a finite join semilattice with a $\hat{0}$, it is a lattice (see, e.g., \cite[Proposition~3.3.1]{stanley2012ec1}).

Next, let us explain why $L^{(k)}$ is graded. Let $y_0 < y_1 < y_2 < \cdots < y_m \in L^{(k)}$ be any chain. We will show that we can extend this chain to a chain of length $n$. To that end, let~$L'$ denote the sublattice of $L$ generated by $\{x_0,\ldots,x_n,y_0,\ldots,y_m\}$. An important result of Stanley~\cite[Proposition 2.1]{stanley1972supersolvable} says that $L'$ is a distributive lattice (in fact, this leads to an alternative characterization of supersolvable lattices). Since~$L'$ is distributive, every element in $L'$ can be written as a join of elements of the form $x_i \wedge y_j$. Hence, Claims 1 and 2 combine to imply that $L'\subseteq L^{(k)}$. Then, again since $L'$ is distributive and hence graded, the chain $y_0 < y_1 < y_2 < \cdots < y_m$ can be extended in $L'$ (and thus in $L^{(k)}$) to a maximal chain of length $n$. So indeed $L^{(k)}$ is graded, and in fact each element in $L^{(k)}$ has the same rank as it does in $L$.

Next, we show that $L^{(k)}$ is semimodular. Let $x,y \in L^{(k)}$. As mentioned, the join of $x$ and $y$ in $L^{(k)}$ is the same as its join in $L$. On the other hand, the meet of $x$ and~$y$ in $L^{(k)}$ can only be lower than in $L$. So, the quantity $\rho(x \vee y) + \rho(x \wedge y)$ is smaller when computed in $L^{(k)}$ versus $L$. Thus, the inequality $\rho(x) + \rho(y) \geq \rho(x \vee y) + \rho(x \wedge y)$ remains true in $L^{(k)}$, and $L^{(k)}$ is semimodular. 

Finally, we explain why $L^{(k)}$ is supersolvable, with maximal chain of modular elements $x_0 \lessdot x_1 \lessdot \cdots \lessdot x_n$. Let $y \in L^{(k)}$ and $i \in \{0,1,\ldots,n-1\}$. Claims 1 and 2 imply that $x_i \vee y$ and $x_i \wedge y$ are the same in $L^{(k)}$ as they are in $L$, so  $\rho(x_i) + \rho(y) = \rho(x_i \vee y) + \rho(x_i \wedge y)$ remains true, and thus $x_i$ is modular in $L^{(k)}$.
\end{proof}

\subsection{Uniform sequences of geometric lattices} \label{subsec:uniform_seqs}

A sequence of finite posets $P_0, P_1, P_2, \ldots$ is a \dfn{uniform sequence} if (for each $n=0,1,\ldots$):
\begin{itemize}
\item $P_n$ is $n$-graded;
\item $[a,\hat{1}_{P_n}]$ is isomorphic to $P_{n-1}$ for all atoms $a \in P_n$.
\end{itemize}
From now on in this section, $L_0, L_1, \ldots$ is a uniform sequence of geometric lattices. 

Uniform sequences also enjoy many remarkable enumerative properties, as first observed by Dowling~\cite{dowling1973class}. For $0\leq j \leq i$, define the numbers $V(i,j)$ and $v(i,j)$ by
\[ \sum_{j=0}^{i} V(i,j) x^{i-j} = F(L_i; x); \qquad  \sum_{j=0}^{i} v(i,j) x^{i-j} = \chi^*(L_i; x).\]
These numbers $V(i,j)$ and $v(i,j)$ are called the Whitney numbers of the second and first kind, respectively, for our sequence of lattices $L_n$. By convention, let us also declare $V(i,j) \coloneqq v(i,j) \coloneqq 0$ for $j > i$.

\begin{thm}[{Dowling~\cite[Theorem 6]{dowling1973class}; see also~\cite[Exercise~3.130]{stanley2012ec1}}] \label{thm:uniform_dowling}
The infinite, lower unitriangular matrices $[V(i,j)]_{0\leq i,j\leq \infty}$ and $[v(i,j)]_{0 \leq i,j \leq \infty}$ are inverses.
\end{thm}

\begin{example}
Taking $L_n = \Pi_{n+1}$ gives a uniform sequence of geometric lattices. In this case, we have $V(i,j)=S(i+1,j+1)$ and $v(i,j)=s(i+1,j+1)$, the Stirling numbers of the second and first kind. Thus, the two infinite, lower unitriangular matrices in \cref{thm:uniform_dowling} are:
{\small \[ [S(n,k)]_{1\leq n,k \leq \infty} = \begin{pmatrix} 1 & 0 & 0 & 0 & \cdots \\ 1 & 1 & 0 & 0 & \cdots \\ 1 & 3 & 1 & 0 & \cdots \\ 1 & 7 & 6 & 1 & \cdots \\ \vdots & \vdots & \vdots & \vdots & \ddots \end{pmatrix} \textrm{ and } [s(n,k)]_{1\leq n,k \leq \infty} = \begin{pmatrix} 1 & 0 & 0 & 0 & \cdots \\ -1 & 1 & 0 & 0 & \cdots \\ 2 & -3 & 1 & 0 & \cdots \\ -6 & 11 & -6 & 1 & \cdots \\ \vdots & \vdots & \vdots & \vdots & \ddots \end{pmatrix}.\]}
It is a classical result that these matrices of Stirling numbers are inverses.
\end{example}

When the geometric lattices $L_n$ in our uniform sequence are also supersolvable, \cref{thm:super_stanley,thm:uniform_dowling} combine to yield a very strong enumerative corollary, as we now explain. First, we need a preparatory result about symmetric polynomials. Recall that the complete homogeneous and elementary symmetric polynomials in variables $x_1,\ldots,x_n$ are, respectively, given by
\begin{align*}
h_k(x_1,\ldots,x_n) &\coloneqq \sum_{1 \leq i_1 \leq i_2 \leq \cdots \leq i_k \leq n} x_{i_1}x_{i_2}\cdots x_{i_k}; \\
e_k(x_1,\ldots,x_n) &\coloneqq \sum_{1 \leq i_1 < i_2 < \cdots < i_k \leq n} x_{i_1}x_{i_2}\cdots x_{i_k},
\end{align*}
for $k > 0$. We also use the conventions $h_0(x_1,\ldots,x_n) \coloneqq e_0(x_1,\ldots,x_n) \coloneqq 1$ and $h_k(x_1,\ldots,x_n) \coloneqq e_k(x_1,\ldots,x_n) \coloneqq 0$ for $k < 0$.

\begin{prop} \label{prop:sym_poly_matrices}
Define infinite, lower unitriangular matrices $A=[a_{i,j}]_{0\leq i,j \leq \infty}$ and $B=[b_{i,j}]_{0 \leq i,j \leq \infty}$ by letting
\begin{align*}
a_{i,j} &\coloneqq h_{i-j}(x_1,\ldots,x_{j+1}); \\
b_{i,j} &\coloneqq (-1)^{i-j}e_{i-j}(x_1,\ldots,x_i),
\end{align*}
for all $0 \leq i,j$. Then $A$ and $B$ are inverses.
\end{prop}

\begin{proof}
Define the matrix $C=[c_{i,j}]_{0\leq i,j \leq \infty}$ by $C = B A$. Clearly, $c_{n,m} =0$ for $m > n$. For $m\leq n$,
\[ c_{n,m} = \sum_{i=0}^{n-m} (-1)^{i}e_i(x_1,\ldots,x_n) h_{n-m-i}(x_1,\ldots,x_{m+1}). \]
If $m=n$, this certainly equals $1$. So now suppose $m < n$. Then $c_{n,m}$ is the coefficient of $z^{n-m}$ in
\[\left(\sum_{k=0}^{\infty} (-1)^k e_k(x_1,\ldots,x_n) z^k \right) \left(\sum_{k=0}^{\infty} h_k(x_1,\ldots,x_{m+1}) z^k \right) = \prod_{i=1}^{n} (1-x_iz) \prod_{i=1}^{m+1} \frac{1}{1-x_iz},\]
which is $0$ (since the right-hand side is a polynomial in $z$ of degree $n-m-1$). Therefore, $C$ is the identity matrix, and $A$ and $B$ are inverses.
\end{proof}

We return to our uniform sequence of geometric lattices $L_0,L_1,\ldots$. Set 
\[a_n \coloneqq \#\{\textrm{atoms }s\in L_n\} - \#\{\textrm{atoms }s\in L_{n-1}\},\]
for $n=1,2,\ldots$.

\begin{cor} \label{cor:sym_poly_lattices}
Suppose that all the geometric lattices $L_n$ in our uniform sequence are supersolvable. Then their Whitney numbers of the second and first kind are
\begin{align*}
V(i,j) &= h_{i-j}(a_1,\ldots,a_{j+1}); \\
v(i,j) &= (-1)^{i-j}e_{i-j}(a_1,\ldots,a_i),
\end{align*}
for $0 \leq j \leq i$. In particular,  $\chi^*(L_n;x) = (1-a_1x)(1-a_2x)\cdots (1-a_n x)$ for all $n$.
\end{cor}

\begin{proof}
Fix $n \geq 1$. Our first goal is to show $\chi^*(L_n;x) = (1-a_1x)(1-a_2x)\cdots (1-a_n x)$. So let $\hat{0}=x_0 \lessdot x_1 \lessdot \cdots \lessdot x_n=\hat{1}$ be a maximal chain of modular elements in~$L_n$. Consider the modular coatom $x_{n-1}$. Because $\hat{1}$ is the join of the atoms in $L_n$ (since~$L_n$ is atomic), there must be some atom $s\in L_n$ with $s \not \leq x_{n-1}$. This $s$ is a complement of $x_{n-1}$, i.e., $s \vee x_{n-1} = \hat{1}$ and $s \wedge x_{n-1} = \hat{0}$. Since $x_{n-1}$ is modular, there is thus a canonical isomorphism from $[s,\hat{1}]$ to $[0,x_{n-1}]$ given by $x \mapsto x \wedge x_{n-1}$ (see~\cite[\S4.2]{birkhoff1967lattice}). But, by assumption, $[s,\hat{1}]$ is isomorphic to $L_{n-1}$, so $[0,x_{n-1}]$ is isomorphic to $L_{n-1}$ as well. Then by induction, each $[0,x_i]$ is isomorphic to $L_i$. Thus, $\#\{\textrm{atoms } s \in L_n \colon s \leq x_i, s \not \leq x_{i-1}\}= a_i$. So by \cref{thm:super_stanley}, we indeed have $\chi^*(L_n;x) = (1-a_1x)(1-a_2x)\cdots (1-a_n x)$.

The previous paragraph tells us that $v(i,j) = (-1)^{i-j}e_{i-j}(a_1,\ldots,a_i)$ for all $i,j$. We know from \cref{thm:uniform_dowling} that the matrices $[V(i,j)]$ and $[v(i,j)]$ are inverses. But we also know from \cref{prop:sym_poly_matrices} that the inverse of the matrix $[v(i,j)]$ is the matrix $[h_{i-j}(a_1,\ldots,a_{j+1})]$. We conclude $V(i,j) = h_{i-j}(a_1,\ldots,a_{j+1})$ for all $i,j$.
\end{proof}

\begin{example}
Let us continue our running example with $L_n = \Pi_{n+1}$. As we saw earlier, these partition lattices are supersolvable and we have $a_n = n$ in this case. So \cref{cor:sym_poly_lattices} tells us that
\[S(n,k) = h_{n-k}(1,2,\ldots,k); \qquad s(n,k) = (-1)^{n-k} e_{n-k}(1,2,\ldots,n-1).\]
These are classical formulas for the Stirling numbers.
\end{example}

\subsection{Uniform sequences of supersolvable geometric lattices}

Continue to fix a uniform sequence of geometric lattices $L_0, L_1,\ldots$. Observe that, by induction, the interval $[x,\hat{1}]\subseteq L_n$ is isomorphic to $L_{n-\rho(x)}$ for all~$x \in L_n$. So, in a sense, $L_n$ is as close to being upho as a finite graded lattice can be: the principal order filters for all elements \emph{of the same rank} are isomorphic. It is therefore reasonable to try to build an upho lattice by taking a limit of the $L_n$ in some way. This is indeed what we will do, but to do it correctly requires some technical precision.

As we have already hinted, in order to take a limit of the $L_n$ we will need to combine the notion of uniformity with that of supersolvability. But we will also need to make sure that the \emph{way} our lattices are supersolvable is compatible with the way they form a uniform sequence. We saw in the proof of \cref{cor:sym_poly_lattices} that when the~$L_n$ are supersolvable there are many isomorphic copies of $L_{n-1}$ sitting inside of~$L_n$: for each atom $s \in L_n$, the ``upper'' interval~$[s,\hat{1}_{L_n}]$ is isomorphic to $L_{n-1}$, and also for each modular coatom $t \in L_n$, the ``lower'' interval~$[\hat{0}_{L_n},t]$ is isomorphic to $L_{n-1}$. We will need to \emph{fix} all of these isomorphisms and make sure they are compatible with one another.

So, abusing terminology, by a \dfn{uniform sequence of supersolvable geometric lattices} we will mean a uniform sequence of geometric lattices $L_0,L_1,\ldots$ together with (for each $n=0,1,\ldots$):
\begin{itemize}
\item an isomorphism $\theta_s\colon [s,\hat{1}_{L_n}]\to L_{n-1}$ for each atom $s\in L_n$;
\item an embedding (isomorphism onto its image) $\iota_n\colon L_n\to L_{n+1}$,
\end{itemize}
satisfying (for each $n=0,1,\ldots$):
\begin{itemize}
\item $(\iota_{n-1} \circ \theta_s) (x) = (\theta_{\iota_n(s)} \circ \iota_n)(x)$ for each atom $s\in L_n$ and all $x \in [s,\hat{1}_{L_n}]$;
\item the image of $\iota_n$ is $[\hat{0}_{L_{n+1}},t_{n}]$, where $t_{n}\in L_{n+1}$ is a modular coatom.
\end{itemize}
The requirement that the image of $\iota_n$ is $[\hat{0}_{L_{n+1}},t_{n}]$ implies in particular that all the embeddings $\iota_n$ are rank-preserving, i.e., $\rho(\iota_n(x)) =\rho(x)$ for all $x \in L_n$.

\begin{remark}
Notice how in the definition of uniform sequence we required the ``upper'' intervals $[s,\hat{1}_{L_n}]$ to be isomorphic to $L_{n-1}$ for \emph{all} atoms $s\in L_{n}$, whereas here we only require that there be \emph{some} (modular) coatom $t_{n-1} \in L_n$ for which the ``lower'' interval $[\hat{0}_{L_n},t_{n-1}]$ is isomorphic to $L_{n-1}$. This is a crucial distinction!
\end{remark}

\begin{example} \label{ex:partition_maps}
Let us continue to examine the case $L_n=\Pi_{n+1}$. We can upgrade this sequence to a uniform sequence of supersolvable geometric lattices by defining embeddings $\iota_n\colon \Pi_{n+1} \to \Pi_{n+2}$ and isomorphisms $\theta_s\colon [s,\hat{1}_{\Pi_{n+1}}]\to \Pi_n$ as follows. First of all, we set $\iota_n(\pi) \coloneqq \pi \cup \{\{n+2\}\}$ for all $\pi \in \Pi_{n+1}$. Next, note that any atom $s \in \Pi_{n+1}$ has a unique non-singleton block, of the form $\{i,j\}$ for~$1 \leq i < j \leq n+1$. Let us denote this atom by $s_{i,j}$. Then for $\pi = \{B_1,\ldots,B_m\} \in [s_{i,j},\hat{1}_{\Pi_{n+1}}]$ we set $\theta_{s_{i,j}}(\pi) \coloneqq \bigcup_{\ell=1}^{m} \{f(k) \colon k \in B_{\ell} \setminus \{j\}\},$
where $f(k)$ is $k$ if $k < j$ and $k-1$ if~$k > j$. In other words, to obtain $\theta_{s_{i,j}}(\pi)$ from~$\pi$, we delete $j$ from whichever block it appears in (necessarily together with $i$) and then re-index by subtracting one from all numbers greater than $j$. It is straightforward to verify that these $\iota_n$ and $\theta_s$ satisfy the requirements listed above.
\end{example}

From now on in this section, let us assume moreover that our sequence $L_0, L_1, \ldots$ is a uniform sequence of \emph{supersolvable} geometric lattices. So we now also have fixed embeddings $\iota_n\colon L_n \to L_{n+1}$, and isomorphisms $\theta_s\colon [s,\hat{1}_{L_n}]\to L_{n-1}$ for each atom~$s\in L_n$. First of all, let us justify our terminology by explaining how this additional structure indeed forces the $L_n$ to be supersolvable geometric lattices.

\begin{prop} \label{prop:super_seq}
For $n=0,1,\ldots$, the geometric lattice $L_n$ is supersolvable.
\end{prop}

\begin{proof}
Let $x_i\coloneqq  \iota_{n-1} \circ \cdots \circ \iota_{i+1} (t_{i}) \in L_n$ for $i=0,1,\ldots,n-1$, and $x_n \coloneqq \hat{1}_{L_n}$, where $t_{m} \in L_{m+1}$ is the distinguished modular coatom for which the image of $\iota_m$ is~$[\hat{0}_{L_{m+1}},t_m]$. Since the $t_{m}$ are modular, and the embeddings $\iota_m$ are rank-preserving, repeated application of \cref{prop:mod_trans} shows that $x_0 \lessdot x_1 \lessdot \cdots \lessdot x_n \in L_n$ is indeed a maximal chain of modular elements.
\end{proof}

Henceforth, each supersolvable geometric lattice $L_n$ comes together with the maximal chain of modular elements $x_0 \lessdot x_1 \lessdot \cdots \lessdot x_n$ from the proof of \cref{prop:super_seq}. This lets us, for instance, speak of $\nu\colon L_n \to \mathbb{N}$ and $L^{(k)}_n$, as defined in \cref{subsec:super}. Notice that in addition to preserving rank, the embeddings $\iota_n$ preserve the $\nu$ statistic, i.e., $\nu(\iota_n(x)) = \nu(x)$ for all $x \in L_n$. Therefore, the embeddings $\iota_n$ restrict to embeddings $\iota_n\colon L^{(k)}_n \to L^{(k)}_{n+1}$ for any $k \geq 1$.

\subsection{Limits of uniform sequences of supersolvable geometric lattices} \label{subsec:limit}

Continue to fix a uniform sequence of supersolvable geometric lattices $L_0, L_1, \ldots$. We will now take the limit of this sequence, which we will denote by $\mathcal{L}_{\infty}$. Since we have distinguished embeddings $\iota_n \colon L_n \to L_{n+1}$, it makes sense to define $\mathcal{L}_{\infty} \coloneqq \bigcup_{n=1}^{\infty} L_n$. 

More precisely, we define $\mathcal{L}_{\infty} \coloneqq \displaystyle \lim_{\longrightarrow} L_n$, the direct limit of the directed system formed by the $L_n$ together with the $\iota_n$. That is, we let $\mathcal{L}_{\infty} \coloneqq \bigsqcup_{n=1}^{\infty} L_n / \sim$, the disjoint union of all the finite lattices~$L_n$ modulo the equivalence relation $\sim$ generated by $x \sim \iota_n(x)$ for all $x\in L_n$ and all $n=0,1,\ldots$. Denoting the equivalence class of an element $x \in \bigsqcup_{n=1}^{\infty} L_n$ by $[x]$, the partial order on $\mathcal{L}_{\infty}$ is $[x] \leq [y]$ if $x' \leq y'$ in some~$L_n$ for some $x' \in [x], y' \in [y]$. 

Because the $L_n$ are graded and the embeddings $\iota_n$ are rank-preserving, their limit $\mathcal{L}_{\infty}$ is $\mathbb{N}$-graded. Similarly, since each $L_n$ is a lattice, their limit $\mathcal{L}_{\infty}$ is a lattice. And the uniformity of the sequence $L_n$ can be used to show that every principal order filter in $\mathcal{L}_{\infty}$ is isomorphic to $\mathcal{L}_{\infty}$. But $\mathcal{L}_{\infty}$ is \emph{not} finite type $\mathbb{N}$-graded: for instance, it has infinitely many atoms. To summarize, the limit poset~$\mathcal{L}_{\infty}$ is an upho lattice, except for the fact that it is not finite type $\mathbb{N}$-graded.

To resolve this final wrinkle, we need to make a ``thinner'' poset out of the ``wide'' poset $\mathcal{L}_{\infty}$. This is where the trimming procedure described in \cref{subsec:super} comes into play. So, define $\nu\colon \mathcal{L}_{\infty}\to \mathbb{N}$ by $\nu([x]) \coloneqq \min\{n\colon x' \in L_n \textrm{ for some $x'\in[x]$}\}$ for all $[x]\in \mathcal{L}_{\infty}$. And then for any positive integer $k \geq 1$, define 
\[\mathcal{L}^{(k)}_{\infty} \coloneqq \{ [x] \in \mathcal{L}_{\infty}\colon \nu([x])-\rho([x]) < k\}.\] 
Since $\rho([x]) = \rho(x)$ and $\nu([x]) = \nu(x)$ for any $[x] \in \mathcal{L}_{\infty}$, we equivalently have that $\mathcal{L}^{(k)}_{\infty} = \bigcup_{n=1}^{\infty} L^{(k)}_n$, the direct limit of the trimmed finite lattices $L^{(k)}_n$ with respect to the embeddings $\iota_n\colon L^{(k)}_n \to L^{(k)}_{n+1}$. Yet another way to think of $\mathcal{L}^{(k)}_{\infty}$ is that it consists of all the elements of rank $n-k+1$ in each $L_n$, for~$n \geq k-1$. It is this~$\mathcal{L}^{(k)}_{\infty}$ which is a proper (i.e., finite type $\mathbb{N}$-graded) upho lattice.

\begin{example}
We continue with the example of $L_n=\Pi_{n+1}$. Then $\mathcal{L}_{\infty}$ consists of all partitions of the set~$\{1,2,\ldots\}$ for which all but finitely many blocks are singletons, partially ordered by refinement. This poset $\mathcal{L}_{\infty}$ is $\mathbb{N}$-graded: we have $\rho(\pi)=\sum_{B \in \pi} (\#B - 1)$. But $\mathcal{L}_{\infty}$ is \emph{not} finite type $\mathbb{N}$-graded. Meanwhile, $\mathcal{L}^{(k)}_{\infty}$ can be viewed as the collection of all partitions of a set of the form $[n]=\{1,2,\ldots,n\}$ (for some $n \geq k$) into $k$ blocks. The partitions in $\mathcal{L}^{(k)}_{\infty}$ are not all partitions of the same set, but the partial order can be described in the same way: for $\pi, \pi'\in \mathcal{L}^{(k)}_{\infty}$, we have $\pi \leq \pi'$ if for every~$B \in \pi$ there exists a~$B' \in \pi'$ with $B \subseteq B'$. And now $\mathcal{L}^{(k)}_{\infty}$ \emph{is} finite type $\mathbb{N}$-graded: we have $\rho(\pi) = n-k$ if $\pi$ is a partition of $[n]$, so there are only finitely many elements of each rank. \Cref{fig:partitions} depicts $\mathcal{L}^{(2)}_{\infty}$ for this example.
\end{example}

The following theorem is the main result of this section. (In the introduction, it was stated, less precisely, as \cref{thm:super_intro}.)

\begin{thm} \label{thm:super}
For any $k \geq 1$, $\mathcal{L}^{(k)}_{\infty}$ is an upho lattice whose core is $L_k$.
\end{thm}

\begin{proof}
Let us first explain why $\mathcal{L}^{(k)}_{\infty}$ is a finite type $\mathbb{N}$-graded lattice. As mentioned above, $\mathcal{L}^{(k)}_{\infty} = \bigcup_{n=1}^{\infty} L^{(k)}_n$. We know from \cref{lem:super_trim} that the $L^{(k)}_n$ are graded posets, and the embeddings $\iota_n$ are rank-preserving, so the limit $\mathcal{L}^{(k)}_{\infty}$ is $\mathbb{N}$-graded. Similarly, we know from \cref{lem:super_trim} that the $L^{(k)}_n$ are lattices, so the limit $\mathcal{L}^{(k)}_{\infty}$ is a lattice as well. Finally, as mentioned, for each $n \geq k-1$, the elements of rank $n-k +1$ in~$\mathcal{L}^{(k)}_{\infty}$ are precisely $[x]$ for $x \in L_n$ with $\rho(x)=n-k+1$. Hence, there are only finitely many elements of each rank in $\mathcal{L}^{(k)}_{\infty}$, i.e., $\mathcal{L}^{(k)}_{\infty}$ is finite type $\mathbb{N}$-graded.

Next, let us explain why $\mathcal{L}^{(k)}_{\infty}$ is upho. Actually, it is convenient to first work with~$\mathcal{L}_{\infty}$. We will show that every principal order filter in $\mathcal{L}_{\infty}$ is isomorphic to $\mathcal{L}_{\infty}$. It clearly suffices to do this for principal order filters $V_{[s]} \subseteq \mathcal{L}_{\infty}$ corresponding to atoms~$[s]\in \mathcal{L}_{\infty}$. So let $[s]\in \mathcal{L}_{\infty}$ be an atom. Recall the distinguished modular coatoms $t_n \in L_{n+1}$ for which $\iota_n(L_n) = [\hat{0}_{L_{n+1}},t_n]$. Let $n_0 \coloneqq \nu([s])$ and notice that~$n_0$ is also the smallest $n_0$ for which $[s] \leq [t_{n_0}]$. We will define a series of isomorphisms $\eta_n\colon [[s],[t_{n}]] \to [\hat{0},[t_{n-1}]]$, for $n \geq n_0$, with isomorphism $\eta\colon V_{[s]} \to \mathcal{L}_{\infty}$ then obtained as the limit of $\eta_n$.

To define $\eta_n$, assume we have chosen the representative $s$ of the equivalence class~$[s]$ so that $s \in L_{n}$. Then the elements of $[[s],[t_{n}]]$ are $[x]$ for $x \in [s,\hat{1}_{L_{n}}]$. By supposition, we have an isomorphism $\theta_s\colon [s,\hat{1}_{L_{n}}] \to L_{n-1}$. But also, the elements of~$[\hat{0},[t_{n-1}]]$ are $[x]$ for $x \in L_{n-1}$. The isomorphism $\eta_n\colon [[s],[t_{n}]] \to [\hat{0},[t_{n-1}]]$ is thus defined by composing $\theta_s$ with the identifications of $[[s],[t_{n}]]$ and $[s,\hat{1}_{L_n}]$, and of~$L_{n-1}$ and $[\hat{0},[t_{n-1}]]$. Crucially, the compatibility requirement $\iota_{n-1} \circ \theta_s = \theta_{\iota_n(s)} \circ \iota_n$ implies that the restriction of $\eta_{n+1}$ to the domain of $\eta_n$ agrees with $\eta_n$.

So we get a sequence of isomorphisms $\eta_n\colon [[s],[t_{n}]] \to [\hat{0},[t_{n-1}]]$, for $n \geq n_0$, which have the property that the restriction of $\eta_n$ to the domain of $\eta_m$ agrees with~$\eta_m$ for all~$m \leq n$. Moreover, each $[x] \in V_{[s]}$ (respectively, $[x] \in \mathcal{L}_{\infty}$) belongs to~$[[s],[t_{n}]]$ (resp.,~$[\hat{0},[t_{n-1}]]$) for sufficiently large~$n$. It therefore makes sense to define the isomorphism $\eta \colon V_{[s]} \to \mathcal{L}_{\infty}$ by $\eta \coloneqq \bigcup_{n=n_0}^{\infty}\eta_n$, i.e., $\eta([x]) \coloneqq \eta_n([x])$ if $[x]$ belongs to the domain of~$\eta_n$.

Now we return to $\mathcal{L}^{(k)}_{\infty}$. We want to show that every principal order filter in~$\mathcal{L}^{(k)}_{\infty}$ is isomorphic to $\mathcal{L}^{(k)}_{\infty}$, and again it is sufficient to only consider the principal order filters for atoms. So let $[s] \in \mathcal{L}^{(k)}_{\infty}$ be an atom, and consider the corresponding principal order filter, which for clarity we denote by $V^{(k)}_{[s]} \subseteq \mathcal{L}^{(k)}_{\infty}$. We claim that the isomorphism $\eta\colon V_{[s]} \to \mathcal{L}_{\infty}$ above restricts to an isomorphism~$\eta\colon V^{(k)}_{[s]} \to \mathcal{L}^{(k)}_{\infty}$. To see this, observe that, for each $n \geq k-1$, the elements of rank $n-k+1$ in~$V^{(k)}_{[s]}$ are the elements of rank $n-k+1$ in $[[s],[t_{n+1}]]$. By construction, $\eta_{n+1}$ maps these elements to the elements of rank $n-k+1$ in $[\hat{0},[t_{n}]]$, which are precisely the elements of rank $n-k+1$ in $\mathcal{L}^{(k)}_{\infty}$. So indeed, $\eta$ restricts to an isomorphism~$\eta\colon V^{(k)}_{[s]} \to \mathcal{L}^{(k)}_{\infty}$.

Finally, to see why the core of $\mathcal{L}^{(k)}_{\infty}$ is $L_k$, notice that the atoms of $\mathcal{L}^{(k)}_{\infty}$ are $[s]$ for atoms $s \in L_k$. Joins in $\mathcal{L}^{(k)}_{\infty}$ are joins of the representative elements in the~$L_n$, so the join of the atoms of $\mathcal{L}^{(k)}_{\infty}$ is $[t_k]$, and its core is $[\hat{0},[t_k]]$, i.e., a copy of $L_k$.
\end{proof}

\begin{cor} \label{cor:super}
For any $k \geq 1$, we have 
\[ F(\mathcal{L}^{(k)}_{\infty}; x) = \frac{1}{(1-a_1x)(1-a_2x)\cdots(1-a_kx)}.\]
\end{cor}

\begin{proof}
We know that $\chi^*(L_k;x)=(1-a_1x)(1-a_2x)\cdots(1-a_kx)$ (see~\cref{cor:sym_poly_lattices}). The corollary then follows from \cref{thm:super,cor:upho_lat_rgf}.
\end{proof}

\begin{remark}
I thank David Speyer and Gjergji Zaimi for explaining the following to me in answers to a question I posted to MathOverflow~\cite{hopkins2022MO}. Suppose that~$A=[a_{ij}]_{0 \leq i,j \leq \infty}$ and $B=[b_{ij}]_{0 \leq i,j \leq \infty}$ are two infinite, lower unitriangular matrices satisfying:
\begin{itemize}
\item $A$ and $B$ are inverses;
\item $\displaystyle \sum_{i=k}^{\infty} a_{i,k} \; x^{i-k} = \left(\sum_{i=0}^{k+1} b_{k+1,k+1-i} \; x^i \right)^{-1}$ for all $k \geq 0$.
\end{itemize}
Let $a_k \coloneqq a_{k,k-1} - a_{k-1,k-2}$ for all $k \geq 1$ (where by convention $a_{0,-1} \coloneqq 0$). Then the entries of $A$ and $B$ are determined by this sequence $a_1,a_2,\ldots$. Specifically, we must have that 
\begin{align*}
a_{i,j} &= h_{i-j}(a_1,\ldots,a_{j+1}); \\
b_{i,j} &= (-1)^{i-j}e_{i-j}(a_1,\ldots,a_i).
\end{align*}
Because of this, we could alternatively have deduced \cref{cor:super} by combining \cref{thm:super,cor:upho_lat_rgf} with Dowling's \cref{thm:uniform_dowling}, without any appeal to Stanley's \cref{thm:super_stanley}.
\end{remark}

\subsection{Examples of uniform sequences of supersolvable geometric lattices}

The construction presented above in this section would not be interesting if there were not any interesting examples of uniform sequences of supersolvable geometric lattices. Fortunately, there are many interesting examples. We now review all examples that we know of.

\subsubsection{Boolean lattices} \label{subsec:bool}

Let $L_n = B_n$ be the sequence of Boolean lattices. 

We define the auxiliary data~$\iota_n\colon B_n\to B_{n+1}$ and $\theta_{s}\colon [s, \hat{1}_{B_n}]\to B_{n-1}$ as follows. First, the embedding $\iota_n$ comes from the inclusion of sets $[n]\subseteq [n+1]$. That is, we set~$\iota_n(S)\coloneqq S$ for all~$S\in B_n$. Next, noting that any atom in $B_n$ has the form~$\{i\}$ for $1 \leq i \leq n$, we set~$\theta_{\{i\}}(S) \coloneqq \{f(j) \colon j \in S \setminus \{i\}\}$ for all $S \in [\{i\}, \hat{1}_{B_n}]$, where $f(j)$ is $j$ if~$j < i$ and~$j-1$ if~$j > i$. In other words, to obtain $\theta_{\{i\}}(S)$ from $S$, we delete $i$ and then re-index by subtracting one from all numbers greater than $i$. 

It is straightforward to verify that this gives a uniform sequence of supersolvable geometric lattices. Here $a_n = 1$ for all $n$, and $V(i,j)=\binom{i}{j}$ and~$v(i,j)=(-1)^{i-j}\binom{i}{j}$. 

Let us use the notation $\mathcal{B}_{\infty} \coloneqq \mathcal{L}_{\infty}$ and $\mathcal{B}^{(k)}_{\infty}\coloneqq \mathcal{L}^{(k)}_{\infty}$ for this sequence $L_n = B_n$. Then, $\mathcal{B}_{\infty}$ is the poset of finite subsets of~$\{1,2,\ldots\}$, ordered by inclusion. And for any $k \geq 1$, $\mathcal{B}^{(k)}_{\infty} = \{\textrm{finite } S\subseteq \{1,2,\ldots\}\colon \max(S) < \#S + k\}$, ordered by inclusion. Equivalently, $\mathcal{B}^{(k)}_{\infty} = \{\textrm{finite } S\subseteq \{1,2,\ldots\}\colon S \subseteq [\#S + k - 1]\}$. Since the core of this upho lattice $\mathcal{B}^{(k)}_{\infty}$ is $B_k$, we have $F(\mathcal{B}^{(k)}_{\infty};x) = 1/(1-x)^{k}$.

Although the core of $\mathcal{B}^{(k)}_{\infty}$ is $B_k$, we note that $\mathcal{B}^{(k)}_{\infty}$ is \emph{not} isomorphic to $\mathbb{N}^k$ (for any~$k \geq 2$). Indeed, the element $\{1,2,\ldots,n\}$ in $\mathcal{B}^{(k)}_{\infty}$ covers $n$ elements. But in $\mathbb{N}^{k}$, each element covers at most $k$ elements. This gives the simplest example showing that an upho lattice is not determined by its core.

\subsubsection{Subspace lattices}

Now fix a prime power $q$, and let $L_n =  B_n(q)$ be the sequence of subspace lattices over $\mathbb{F}_q$. In order to define the various maps between the~$B_n(q)$, it will be necessary to concretely represent elements of the vector spaces~$\mathbb{F}_q^n$. Thus, let each $\mathbb{F}_q^n$ come with a distinguished basis $\{e_1,\ldots,e_n\}$. This gives a canonical inclusion $\mathbb{F}_q^{n} \subseteq \mathbb{F}_q^{n+1}$. It also means that each subspace $U \in  B_n(q)$ can be represented uniquely by the matrix in reduced column echelon form whose column space is $U$. Similarly, elements $g \in \mathrm{GL}(\mathbb{F}_q^n)$ of the general linear group of~$\mathbb{F}_q^n$ can also now be represented by matrices. These matrix representations are useful because we can order matrices lexically, by reading them row-by-row.

We define the auxiliary data~$\iota_n\colon B_n(q)\to B_{n+1}(q)$ and $\theta_{s}\colon [s, \hat{1}_{B_n(q)}]\to B_{n-1}(q)$ as follows. First, $\iota_n$ comes from the inclusion of vector spaces $\mathbb{F}_q^{n} \subseteq \mathbb{F}_q^{n+1}$. That is, to get the representing matrix of $\iota_n(U)$ from that of $U \in B_n(q)$, we append a row of zeros. Next, noting that any atom ($1$-dimensional subspace) $S \in B_n(q)$ has several complementary $(n-1)$-dimensional subspaces, we choose the complement~$T \in B_n(q)$ of $S$ whose representing matrix is first in lexical order. Similarly, there are several $g \in \mathrm{GL}(\mathbb{F}_q^n)$ with $g \cdot T = \mathrm{Span}\{e_1,\ldots,e_{n-1}\}$, so we choose the $g$ whose representing matrix is first in lexical order. Then, set $\theta_S(U) \coloneqq \iota_{n-1}^{-1} (g(T\cap U))$ for $U \in [S, \hat{1}_{B_n(q)}]$. 

It is again a straightforward check that we get a uniform sequence of supersolvable geometric lattices. Here $a_n = q^n$, $V(i,j) = \qbinom{i}{j}_q$ and $v(i,j) = (-1)^{i-j}q^{\binom{i-j}{2}}\qbinom{i}{j}_q$.

We use notation $\mathcal{B}_{\infty}(q) \coloneqq \mathcal{L}_{\infty}$ and $\mathcal{B}^{(k)}_{\infty}(q) \coloneqq \mathcal{L}^{(k)}_{\infty}$ for this sequence $L_n = B_n(q)$. Let $\mathbb{F}_q^{\infty}$ denote the (infinite-dimensional) $\mathbb{F}_q$-vector space with basis $\{e_1,e_2,\ldots\}$. Then, $\mathcal{B}_{\infty}(q)$ is the poset of finite-dimensional subspaces of~$\mathbb{F}_q^{\infty}$, ordered by inclusion. And $\mathcal{B}^{(k)}_{\infty}(q) = \{\textrm{finite-dimensional } U\subseteq \mathbb{F}_q^{\infty}\colon U \subseteq \mathrm{Span}(\{e_1,e_2,\ldots,e_{\dim(U) + k -1}\})\}$  for any $k \geq 1$. We have $F(\mathcal{B}^{(k)}_{\infty}(q);x) = 1/((1-x)(1-qx) \cdots (1-q^{n-1}x))$, since the core of $\mathcal{B}^{(k)}_{\infty}(q)$ is $B_k(q)$. Evidently, $\mathcal{B}^{(k)}_{\infty}(q)$ is a $q$-analogue of $\mathcal{B}^{(k)}_{\infty}$.

\subsubsection{Partition lattices}

The partition lattices were our running example above, but for completeness we repeat everything here. It is well-known that the partition lattices $\Pi_{n}$ are supersolvable geometric lattices, and taking $L_n = \Pi_{n+1}$ gives a uniform sequence. The appropriate auxiliary data $\iota_n$ and $\theta_s$ were defined in \cref{ex:partition_maps}. As a reminder, $\iota_n \colon \Pi_{n+1} \to \Pi_{n+2}$ is given by $\iota_n(\pi) = \pi \cup \{ \{n+2\} \}$. Here~$a_n = n$, and $V(i,j)=S(i+1,j+1)$ and $v(i,j)=s(i+1,j+1)$, the Stirling numbers of the second and first kind.

Let us use the notation $\prod_{\infty} \coloneqq \mathcal{L}_{\infty}$ and $\prod^{(k)}_{\infty} \coloneqq \mathcal{L}^{(k)}_{\infty}$ for this sequence $L_n = \Pi_{n+1}$. Then, $\prod_{\infty}$ is the poset of all partitions of the set~$\{1,2,\ldots\}$ for which all but finitely many blocks are singletons, ordered by refinement. And for any $k \geq 1$, $\prod^{(k)}_{\infty}$ can be identified with the collection of partitions of $[n]$ into $k$ blocks, for some $n \geq k$. The partial order on $\prod^{(k)}_{\infty}$ is still refinement in the sense that for $\pi, \pi' \in \prod^{(k)}_{\infty}$ we have $\pi \leq \pi'$ if for every $B\in \pi$ there exists $B' \in \pi'$ with $B\subseteq B'$. We have $F(\prod^{(k)}_{\infty};x) = 1/((1-x)(1-2x) \cdots (1-kx))$, since the core of $\prod^{(k)}_{\infty}$ is $\Pi_{k+1}$. As mentioned above, \cref{fig:partitions} depicts $\prod^{(2)}_{\infty}$.

\subsubsection{Type~B partition lattices}

We now describe a Type~B variant of the previous example. For an integer $i \in \mathbb{Z}$, let us use the shorthand $\overline{i} \coloneqq -i$, and for a subset of integers $S\subseteq \mathbb{Z}$, let $\overline{S} \coloneqq \{\overline{i}\colon i \in S\}$. For $n \geq 1$, a partition $\pi$ of the set $[n] \cup \overline{[n]}$ is called a Type~B partition if:
\begin{itemize}
\item for every block $B \in \pi$, we also have $\overline{B} \in \pi$;
\item there is at most one block $B \in \pi$ (called the zero block) with $B = \overline{B}$.
\end{itemize}
The \dfn{Type~B partition lattice} $\Pi^B_{n}$ is the poset of Type~B partitions of $[n] \cup \overline{[n]}$, ordered by refinement. 

It is well-known that $\Pi^B_{n}$ is a geometric lattice. In fact, just as $\Pi_{n+1}$ is the lattice of flats of the Coxeter arrangement of Type~$A_n$, $\Pi^B_{n}$ is the lattice of flats of the Coxeter arrangement of Type~$B_n$ (see~\cite{zaslavsky1981geometry,reiner1997noncrossing,stanley2007hyperplane}). Moreover, it is known that $\Pi^B_{n}$ is supersolvable, and taking $L_n = \Pi^B_n$ gives a uniform sequence. (This also follows from a more general result of Dowling~\cite{dowling1973class} discussed below.) Note that the rank function on $\Pi^B_{n}$ is given by $\rho(\pi) = n-k$ if $\pi$ has $2k$ non-zero blocks.

We define $\iota_n\colon \Pi^B_{n} \to \Pi^B_{n+1}$ and~$\theta_s\colon [s,\hat{1}_{\Pi^B_{n}}]\to \Pi^B_{n-1}$ for this sequence as follows. First, we set $\iota_n(\pi) \coloneqq \pi \cup \{ \{n+1\}, \{\overline{n+1}\}\}$ for all~$\pi \in \Pi^B_{n}$. Next, consider an atom~$s \in \Pi^B_{n}$. The atom $s$ could have a single non-singleton block of the form~$\{i,\overline{i}\}$ for~$1 \leq i \leq n$; denote this kind of atom by~$s_i$. Or, the atom $s$ could have two non-singleton blocks of the form $\{i,j\}, \{\overline{i},\overline{j}\}$ (respectively, $\{i,\overline{j}\}, \{\overline{i},j\}$) for~$1 \leq i < j \leq n$; denote this kind of atom by $s_{i,j}$ (resp., $s_{i,\overline{j}}$). In order to not get bogged down by notation, let us describe how to obtain $\theta_s(\pi)$ from $\pi$ in words only. In the case where $s=s_i$, we obtain $\theta_s(\pi)$ from $\pi$ by deleting $i$ and $\overline{i}$ from $\pi$, and re-indexing by decreasing by one the absolute value of numbers greater in absolute value than~$i$.  In the case where $s=s_{i,j}$ or $s_{i,\overline{j}}$, we obtain $\theta_s(\pi)$ from~$\pi$ by deleting $j$ and $\overline{j}$ from~$\pi$, and re-indexing by decreasing by one the absolute value of numbers greater in absolute value than $j$. These are slight variations of the maps we used for the Type~A partition lattices, and it is again a straightforward, albeit tedious, check that they satisfy the requirements.

For this sequence, $a_n = 2n-1$, and $V(i,j) = S_B(i,j)$ and $v(i,j)=s_B(i,j)$, where $S_B(n,k)$ and $s_B(n,k)$ are the Type~B Stirling numbers of the second and first kind. These Type~B Stirling numbers are defined by
\[S_B(n,k) \coloneqq h_{n-k}(1,3,\ldots,2k+1); \qquad s_B(n,k) \coloneqq (-1)^{n-k}e_{n-k}(1,3,\ldots,2n-1).\]
(See the recent paper~\cite{sagan2022stirling} for more about the Type~B Stirling numbers.) In particular, we have that $\chi^*(\Pi^{B}_n;x) = (1-x)(1-3x)\cdots(1-(2n-1)x)$.

We use the notation $\prod^{B}_{\infty} \coloneqq \mathcal{L}_{\infty}$ and $\prod^{B,(k)}_{\infty} \coloneqq \mathcal{L}^{(k)}_{\infty}$ for this sequence $L_n = \Pi^{B}_{n}$. Let us call a partition of the set~$\mathbb{Z} \setminus \{0\}$ a Type~B partition if it satisfies the same two conditions in the bulleted list above. Then $\prod^{B}_{\infty}$ is the poset of all Type~B partitions of $\mathbb{Z} \setminus \{0\}$ where all but finitely many blocks are singletons, ordered by refinement. And for any $k \geq 1$, we can view $\prod^{B,(k)}_{\infty}$ as the collection of Type~B partitions of a set of the form $[n]\cup\overline{[n]}$ (for $n \geq k - 1$) which have $2(k-1)$ non-zero blocks. The partial order on $\prod^{B,(k)}_{\infty}$ is still refinement in the sense that for $\pi, \pi' \in \prod^{B,(k)}_{\infty}$ we have $\pi \leq \pi'$ if for every $B\in \pi$ there exists $B' \in \pi'$ with $B\subseteq B'$. We have $F(\prod^{B,(k)}_{\infty};x) = 1/((1-x)(1-3x) \cdots (1-(2k-1)x))$, since the core of $\prod^{B,(k)}_{\infty}$ is $\Pi^{B}_{k}$.

\subsubsection{Dowling lattices} \label{subsec:dowling}

The most sophisticated example of a uniform sequence of supersolvable geometric lattices is due to Dowling~\cite{dowling1973qanalog,dowling1973class}. The \dfn{Dowling lattice}~$Q_n(G)$ depends on the choice of a finite group $G$. Choosing $G$ to be the trivial group gives~$Q_n(G) = \Pi_{n+1}$, and choosing $G=\mathbb{Z}/2\mathbb{Z}$ gives~$Q_n(G) =\Pi^B_{n}$. In this way, the Dowling lattices recover the previous two examples as special cases. Dowling first defined $Q_n(G)$ in~\cite{dowling1973qanalog} for $G$ the multiplicative group of a finite field, and then in~\cite{dowling1973class} for any finite group $G$. See also~\cite[\S5.3]{doubilet1972foundations} for a graphical description of $Q_n(G)$.

We now review the construction of Dowling lattices. Thus, fix a finite group $G$, say with $m$ elements. The construction of $Q_n(G)$ requires several technical definitions, so please bear with us.

A partial partition of a set~$X$ is a collection~$\pi =\{B_1,\ldots,B_k\}$ of nonempty subsets of $X$ ($\emptyset \subsetneq B_1,\ldots,B_k \subseteq X$) that are pairwise disjoint ($B_i \cap B_j = \emptyset$ for $i \neq j$). In other words, a partial partition of $X$ is a partition of a subset of $X$. Note that we allow $k=0$, i.e., the partition with no blocks. There is a canonical bijection between the partial partitions of $[n]$ and the (usual) partitions of $[n+1]$ which takes the partial partition $\pi = \{B_1,\ldots,B_k\}$ to $\pi' \coloneqq \pi \cup \{ [n+1]\setminus \cup_{i=1}^{k} B_i\}$.

A $G$-labeled set is a map $\alpha\colon A \to G$ from a set $A$ to~$G$. We also denote such a $G$-labeled set by the pair $(\alpha,A)$. We say that two $G$-labeled sets~$(\alpha,A)$ and~$(\beta,B)$ are equivalent if $A=B$ and there is $g \in G$ such that $\alpha(x) = g \cdot \beta(x)$ for all $x \in A$. We denote the equivalence class of $(\alpha,A)$ by $[\alpha,A]$. 

A partial $G$-partition of a set~$X$ is a collection $\alpha = \{[\alpha_1,A_1], [\alpha_2,A_2],\ldots,[\alpha_k,A_k]\}$ of equivalence classes of $G$-labeled sets for which the underlying sets $\{A_1,\ldots,A_k\}$ form a partial partition of $X$. We continue to refer to the $A_i$ as the blocks of $\alpha$.

We can now define $Q_n(G)$. The elements of $Q_n(G)$ are all partial $G$-partitions of~$[n]$. And the partial order is: for partial $G$-partitions $\alpha=\{[\alpha_1,A_1],\ldots,[\alpha_k,A_k]\}$ and $\beta=\{[\beta_1,B_1],\ldots,[\beta_{\ell},B_{\ell}]\}$, we have $\alpha \leq \beta$ if
\begin{itemize}
\item each block $B_j$ in $\beta$ is a union $B_j = A_{i_1} \cup \cdots \cup A_{i_r}$ of blocks $A_{i_1},\ldots,A_{i_r}$ in~$\alpha$;
\item for any block $A_i$ in $\alpha$ with $A_i \subseteq B_j$ for some block $B_j$ in $\beta$, we have that the restriction $\beta_{j} \mid_{A_i}\colon A_i \to G$ of the $G$-labeled set $\beta_j\colon B_j \to G$ to $A_i$ is equivalent to the $G$-labeled set $\alpha_i \colon A_i \to G$.
\end{itemize}
For instance, the maximum element of $Q_n(G)$ is the partial $G$-partition $\alpha=\emptyset$ with no blocks. And the minimum element of $Q_n(G)$ is $\alpha= \{ [\ast,\{1\}], [\ast,\{2\}], \ldots, [\ast,\{n\}]\}$, where $\ast$ denotes the map to $G$ which is constantly equal to the identity $e \in G$.

Dowling~\cite{dowling1973class} proved that $Q_n(G)$ is a supersolvable geometric lattice of rank $n$, and that taking $L_n=Q_n(G)$ gives a uniform sequence (see also~\cite[Exercise 3.131]{stanley2012ec1}). Note that the rank function on $Q_n(G)$ is given by $\rho(\alpha) = n-k$ if $\alpha$ has $k$ blocks.

We define $\iota_n\colon Q_{n}(G) \to Q_{n+1}(G)$ and~$\theta_s\colon [s,\hat{1}_{Q_{n}(G)}]\to Q_{n-1}(G)$ for this sequence as follows. First, we set $\iota_n(\alpha) \coloneqq \alpha \cup \{[\ast,\{n+1\}]\}$ for all~$\alpha \in Q_n(G)$. Next, consider an atom $s \in Q_n(G)$. This atom $s$ could have all singleton blocks, being of the form $s=\{ [\ast,\{1\}], \ldots, [\ast,\{n\}]\} \setminus \{[\ast,i]\}$ for some~$1 \leq i \leq n$; denote such an atom by~$s_i$. Or, the~$s$ could have a single non-singleton block of the form $\{i,j\}$ for some~$1 \leq i < j \leq n$; denote such an atom by $s_{i,j}$.\footnote{In this notation, we suppress the choice of the $G$-labelling, but the $G$-labelling is irrelevant.} We describe how to obtain~$\theta_s(\alpha)$ from~$\alpha$ in words. In the case where $s=s_i$, to obtain $\theta_s(\alpha)$ from~$\alpha$ we re-index by subtracting one from all numbers greater than $i$. In the case where~$s=s_{i,j}$, to obtain $\theta_s(\alpha)$ from~$\alpha$ we delete~$j$ and re-index by subtracting one from all numbers greater than $j$. It is again a straightforward, albeit tedious, check that these satisfy the requirements.

For this sequence we have $a_n = 1 + (n-1)m$, and hence
\begin{align*}
V(i,j) &= h_{i-j}(1,1+m,1+2m,\ldots,1+jm);\\
v(i,j)  &= (-1)^{i-j}e_{i-j}(1,1+m,1+2m,\ldots,1+(i-1)m).
\end{align*}
In particular, $\chi^*(Q_n(G);x) = (1-x)(1-(1+m)x) \cdots ( 1-(1+(n-1)m)x)$.

Let us use $\mathcal{Q}_{\infty}(G) \coloneqq \mathcal{L}_{\infty}$ and $\mathcal{Q}^{(k)}_{\infty}(G) \coloneqq \mathcal{L}^{(k)}_{\infty}$ for this sequence $L_n = Q_n(G)$. Then, $\mathcal{Q}_{\infty}(G)$ is the poset of partial $G$-partitions $\alpha=\{[\alpha_1,A_1],[\alpha_2,A_2],\ldots\}$ of the set~$\{1,2,\ldots\}$ for which:
\begin{itemize}
\item all but finitely many blocks $A_1,A_2,\ldots$ are singletons;
\item the union $A_1\cup A_2 \cup \cdots$ is cofinite, i.e., $\{1,2,\ldots\} \setminus \cup_{i=1}^{\infty} A_i$ is finite.
\end{itemize}
The partial order $\alpha \leq \beta$ for such partial $G$-partitions is exactly as described above. 

Unfortunately, the best description of $\mathcal{Q}^{(k)}_{\infty}(G)$ we have is slightly ugly. For~$k \geq 1$, we can represent the elements of $\mathcal{Q}^{(k)}_{\infty}(G)$ by pairs $(\alpha,n)$ where $n \in \mathbb{N}$ is a nonnegative integer and $\alpha$ is a partial $G$-partition of $[n]$ into $k-1$ blocks. The partial order is~$(\alpha,i) \leq (\beta,j)$ if $i \leq j$ and $\alpha \cup \{[\ast,\{i+1\}],\ldots,[\ast,\{j\}]\} \leq \beta$ according to the partial order on partial $G$-partitions described above. Alternatively, we can say that $(\alpha,i) \leq (\beta,j)$ if $i \leq j$ and $\alpha \leq \beta'$, where $\beta'$ is the result of deleting $i+1,\ldots,j$ from $\beta$. We have~$F(\mathcal{Q}^{(k)}_{\infty}(G);x) = 1/((1-x)(1-(1+m)x) \cdots (1-(1+(k-1)m)x))$, since the core of $\mathcal{Q}^{(k)}_{\infty}(G)$ is~$Q_k(G)$.

\section{Upho lattices from monoids} \label{sec:monoid}

In this section we explore an algebraic source of upho lattices: monoids. Already Gao--Guo--Seetharaman--Seidel~\cite[\S5]{gao2020upho} observed that cancellative monoids yield upho posets. (In fact, the connection between enumeration in monoids, especially cancellative monoids, and M\"{o}bius function values appeared much earlier in the work of Cartier and Foata~\cite{cartier1969problemes} on free partially commutative monoids.) But to get an upho \emph{lattice} in this way requires a very special kind of monoid. The most significant source of monoids with the lattice property are the Garside monoids~\cite{dehornoy1999gaussian, dehornoy2015foundations}, whose theory and main examples we review below.

\subsection{Monoid basics}

Here we quickly review the monoid terminology and notation we will need in what follows. We mostly follow the terminology in~\cite{dehornoy1999gaussian}. Recall that a \dfn{monoid} $M=(M,\cdot)$ is a set $M$ equipped with an associative binary product~$\cdot$ that has an identity element $1 \in M$. We often suppress the product symbol $\cdot$ when it is clear from context.

For any set $S$, the \dfn{free monoid} on $S$ is the set of (finite length) words over the alphabet $S$, with the product being the concatenation of words. The identity of the free monoid on $S$ is the empty word. 

A \dfn{presentation} of a monoid $M$ is a way of writing the monoid $M=\langle S \mid R \rangle$ as the quotient of the free monoid on some set $S$ by the relations (of the form $w=w'$, where $w$ and $w'$ are words in the free monoid on $S$) in some set~$R$. More precisely, the elements of $\langle S \mid R \rangle$ are equivalence classes of words in the free monoid on~$S$ under the equivalence relation generated by $x w y \sim x w' y$ for all relations~$w = w'$ in~$R$ and all words~$x,y$. $M$ is \dfn{finitely generated} if it has a presentation~$M=\langle S \mid R \rangle$ with $S$ finite. This is not standard terminology, but let us say that $M$ is \dfn{homogeneously finitely generated} if it has a presentation $M=\langle S \mid R \rangle$ with $S$ finite and with all the relations in $R$ homogeneous. Here we say that a relation $w = w'$ is homogeneous if $\ell(w) = \ell(w')$, where $\ell(w)$ is the length of the word $w$. (Equivalently, the finitely generated monoid $M=\langle S \mid R \rangle$ is homogeneously finitely generated if there is a monoid homomorphism $\varphi\colon M \to \mathbb{N}$ with $\varphi(s)=1$ for each generator~$s\in S$.) For a homogeneously finitely generated monoid $M$, the \dfn{length} $\ell(x)$ of any element $x\in M$ is well-defined as the length of any expression for $x$ as a word in the generators.

Now let $M$ be a monoid. A non-identity element $a \in M \setminus \{1\}$ is an \dfn{atom} if it cannot be written as a nontrivial product, i.e., $a=bc$ implies $b=1$ or~$c=1$. The monoid~$M$ is \dfn{atomic} if it is generated by its atoms. It is \dfn{bounded atomic} if it is atomic and for each $x\in M$ there is a finite upper bound for the length of an expression for $x$ as a product of atoms. Note that in an atomic monoid, the sets which generate the monoid are precisely the sets containing the atoms; hence, if the atomic monoid is finitely generated, then it has finitely many atoms. Also note that a homogeneously finitely generated monoid is bounded atomic.

For two elements $a,b\in M$, we say that $a$ is a \dfn{left divisor} of $b$, and $b$ is a \dfn{right multiple} of $a$, if $a x=b$ for some $x \in M$. We use $\leq_L$ for the preorder of \dfn{left divisibility} on $M$: $a \leq_L b$ means that $a$ is a left divisor of $b$. Of course, we also have the dual notions of \dfn{right divisor} and \dfn{left multiple}, and the \dfn{right divisibility} preorder~$\leq_R$. We will mostly be interested in the left order $\leq_L$.

In the situation we are interested in, the divisibility preorders will actually be partial orders. Specifically, if $M$ is bounded atomic then $\leq_L$ and $\leq_R$ are partial orders. Moreover, in this situation, the atoms of $(M,\leq_L)$, and of $(M,\leq_R)$, in the poset-theoretic sense are the same as the atoms of $M$ in the monoid-theoretic sense. (But the atomic property of a monoid is unrelated to the atomic property of a lattice; we apologize for this conflict in terminology, which unfortunately is standard.)

In fact, if $M$ is finitely generated, then it is bounded atomic if and only if $\leq_L$ is a partial order and each $y \in M$ has only finitely many $x\in M$ with $x \leq_L y$: see \cite[Proposition~2.3]{dehornoy1999gaussian}. We can similarly see that $M$ is homogeneously finitely generated if and only if $(M,\leq_L)$ is a finite type $\mathbb{N}$-graded poset. In this case, the rank of an element $x\in M$ is the same as its length $\ell(x)$.

We say that $M$ is \dfn{left cancellative} if for all $a,b,c\in M$, whenever $ab = ac$ then we have $b=c$. Of course, there is also the dual notion of \dfn{right cancellative}. The monoid is \dfn{cancellative} if it is both left cancellative and right cancellative. We will mostly be interested in left cancellative monoids.

\subsection{Upho posets from cancellative monoids}

We get upho lattices from monoids in the following way.

\begin{lemma}[{c.f.~\cite[Lemma~5.1]{gao2020upho} and~\cite{fu2024upho}}] \label{lem:monoid}
Let $M$ be a homogeneously finitely generated monoid. If $M$ is left cancellative, then $\mathcal{L} \coloneqq (M,\leq_L)$ is an upho poset. If additionally every pair of elements in $M$ have a least (with respect to $\leq_L$) common right multiple, then $\mathcal{L}$ is an upho lattice.\end{lemma}

Note that \cref{lem:monoid} was stated as \cref{lem:monoid_intro} in the introduction.

\begin{proof}[Proof of~\cref{lem:monoid}]
As mentioned, if $M$ is homogeneously finitely generated then $\mathcal{L} \coloneqq (M, \leq_L)$ is a finite type $\mathbb{N}$-graded poset. 

So now let us show that if $M$ is left cancellative, then $\mathcal{L}$ is upho. The proof is given in~\cite[Proof of Lemma~5.1]{gao2020upho}, but it is simple so we repeat it here. Let~$x \in M$. We want to construct an isomorphism $\varphi\colon V_x \to \mathcal{L}$ where $V_x \subseteq \mathcal{L}$ is the principal order filter of $\mathcal{L}$ generated by $x$. Every $y \in V_x$ has the form~$y=xz$ for some~$z\in M$; we then set $\varphi(y) \coloneqq z$. This map is well-defined precisely because~$M$ is left cancellative: if $y=xz$ and $y=xw$, then we must have $w=z$. Once we know that~$\varphi$ is well-defined, it is easy to see that it and its inverse (which is multiplication on the left by $x$) respect left divisibility, so that it gives a poset isomorphism.

Finally, let us show that if least common right multiples of pairs of elements in~$M$ exist then $\mathcal{L}$ is a lattice. The idea is the same as we have seen before: the fact that joins exist also implies meets exist. In more detail, let $x,y \in \mathcal{L}$ be any two elements. By assumption, their join $x \vee y$ exists. The interval $[\hat{0},x \vee y]$ is a finite join semilattice with a $\hat{0}$, so, as we have seen, it is a lattice (see, e.g., \cite[Proposition~3.3.1]{stanley2012ec1}). Hence the meet $x \wedge y$ exists, and so $\mathcal{L}$ is a lattice.
\end{proof}

In the remainder of this section we will use \cref{lem:monoid} to construct upho lattices. We note that the upho lattices we obtain in this section via \cref{lem:monoid} are quite different from the ones we saw in the \cref{sec:super}. For instance, whereas all the characteristic polynomials in \cref{sec:super} factored into linear factors over the integers, that will not be the case for the lattices in this section.

\subsubsection{A monoid for rank two cores}

To produce sophisticated examples of monoids satisfying the conditions of \cref{lem:monoid} will require some deeper theory. But there are a few examples we can produce by hand. 

For each $r \geq 1$, there is a unique finite graded lattice of rank two that has exactly $r$ atoms. This lattice is conventionally denoted $M_r$, and we will follow that convention. But please do not confuse the letter M in $M_r$ for monoid; it stands rather for ``modular.'' We can show that each $M_r$, for $r \geq 2$, is a core of some upho lattice by constructing an appropriate monoid. 

\begin{thm}\label{thm:rank_two}
Let $r \geq 2$ and let $M \coloneqq \langle x_1,\ldots, x_r\mid x_ix_1=x_1^2 \textrm{ for all $i=2,\ldots,r$}\rangle$.\footnote{I thank Benjamin Steinberg for clarifying what the simplest presentation of this monoid is.} Then $M$ satisfies the conditions of \cref{lem:monoid}, so that $\mathcal{L} \coloneqq (M,\leq_L)$ is an upho lattice. Its core is $L \coloneqq M_r$. We have~$F(\mathcal{L};x)^{-1} = \chi^*(L;x) = 1-rx+(r-1)x^2$.
\end{thm}

Note that \cref{thm:rank_two} says all rank two finite graded lattices (with at least two atoms) are cores. In \cref{sec:obstruct} we will see that not all rank three lattices are cores.

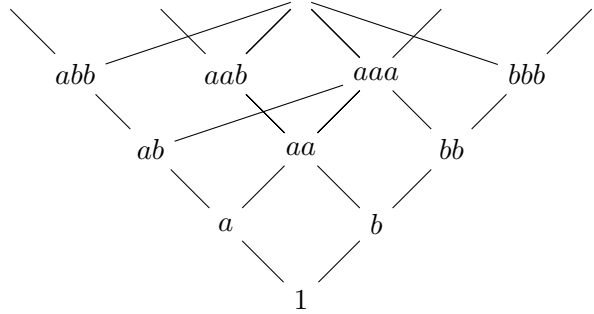
\begin{figure}
\begin{tikzpicture}
\node (A) at (0,0) {$1$};
\node (B) at (-1,1) {$a$};
\node (C) at (1,1) {$b$};
\node (D) at (-2,2) {$ab$};
\node (E) at (0,2) {$aa$};
\node (F) at (2,2) {$bb$};
\node (G) at (-3,3) {$abb$};
\node (H) at (-1,3) {$aab$};
\node (I) at (1,3) {$aaa$};
\node (J) at (3,3) {$bbb$};
\node (K) at (-4,4) { };
\node (L) at (-2,4) { };
\node (M) at (0,4) { };
\node (N) at (2,4) { };
\node (O) at (4,4) { };
\draw (B) -- (A) -- (C);
\draw (D) -- (B) -- (E) -- (C) -- (F);
\draw (G) -- (D) -- (I) -- (E) -- (H) -- (E) -- (I) -- (F) -- (J);
\draw (G) -- (M) -- (H) -- (M) -- (I) -- (M) -- (J);
\draw (G) -- (K);
\draw (H) -- (L);
\draw (I) -- (N);
\draw (J) -- (O);

\end{tikzpicture}
\caption{The monoid $M=\langle a,b\mid ba=aa\rangle$  from the case $r=2$ of \cref{thm:rank_two}, an upho lattice with core $M_2$.} \label{fig:r2}
\end{figure}

\begin{example} \label{ex:r2}
The monoid $M$ from the case $r=2$ of \cref{thm:rank_two} is depicted in \cref{fig:r2}. For ease of display, we have written $a=x_1$ and $b=x_2$ in this figure.
\end{example}

\begin{remark}
Notice that the upho lattice in~\cref{ex:r2} is isomorphic to the lattice denoted $\mathcal{B}^{(2)}_{\infty}$ in~\cref{subsec:bool}. In fact, for any $r\geq 2$, the upho lattice from~\cref{thm:rank_two} is isomorphic to the lattice denoted $\mathcal{L}^{(2)}_{\infty}$ in \cref{subsec:limit}, when $L_0,L_1,\ldots$ is any uniform sequence of supersolvable geometric lattices with $L_2=M_r$. (Such sequences always exist: for $r=2$ we can take the sequence of Boolean lattices, as already discussed; and for $r\geq 3$ we can let $G$ be a group with $r-2$ elements and consider the associated Dowling lattices from \cref{subsec:dowling}.) So, for example, the upho lattice from the case~$r=3$ will be isomorphic to the lattice in \cref{fig:partitions}.
\end{remark}

\begin{remark}
We will study the different ways of realizing the rank two lattices~$M_r$ as cores in much more detail in~\cite{hopkins2024upho2}. There we will show that there are many further monoids which give us many different upho lattices with core $M_r$.
\end{remark}

In fact, we can slightly generalize the construction from \cref{thm:rank_two}, as follows.

\begin{thm}\label{thm:chains}
Let $n, r \geq 2$. Let $M \coloneqq \langle x_1,\ldots, x_r\mid x_ix_1^{n-1}=x_1^{n} \textrm{ for $i=2,\ldots,r$}\rangle$. Then $M$ satisfies the conditions of \cref{lem:monoid}, so that $\mathcal{L} \coloneqq (M,\leq_L)$ is an upho lattice. Its core is $L\coloneqq \{\hat{0}\} \oplus r \cdot [n-1] \oplus\{\hat{1}\}$, i.e., the result of appending a minimum and a maximum to the disjoint union of $r$ $(n-1)$-element chains. We have~$F(\mathcal{L};x)^{-1} = \chi^*(L;x) = 1-rx+(r-1)x^n$.
\end{thm}

The case $n=2$ of \cref{thm:chains} is \cref{thm:rank_two}.

\begin{proof}[Proof of \cref{thm:chains}]
The first thing to notice is that elements of $M$ have a normal form which is easy to describe. Namely, any $u \in M$ can be written, in a unique way, in the form $u=vw$ where $v=x_1^m$ for $m \geq 0$, and $w$ is any word in the~$x_i$ which does not start with an $x_1$ and does not have $x_1^{n-1}$ as a (consecutive) subword. Indeed, to convert a word in the $x_i$ to this normal form, find the last occurrence of~$x_1^{n-1}$ as a subword in this word, and repeatedly use the defining relations to make all preceding letters equal to $x_1$. 

With this normal form in hand, it is straightforward to show that $M$ is left cancellative. We wish to show that whenever we have $a,b,c \in M$ with $b\neq c$ we must have $ab \neq ac$. By considering a counterexample which has $a$ of minimal length, it is enough to consider the case that $a$ is one of the generators $x_i$. Further, we can assume that $b$ and $c$ are written in the normal form. Then, if $i=1$, $x_1b$ and $x_1c$ will be written in normal form, and hence evidently $x_1b \neq x_1c$. On the other hand, if $i \geq 2$, the normal form of $x_ib$ might be $x_1b$, and likewise the normal formal of $x_ic$ might be $x_1c$, but from the fact that $b\neq c$ it is still clear that $x_ib\neq x_ic$.

Now that we have shown $M$ is left cancellative, it follows from~\cref{lem:monoid} that $\mathcal{L} = (M,\leq_L)$ is a (finite type $\mathbb{N}$-graded) upho poset. What remains to show is that~$\mathcal{L}$ is a lattice. As explained in \cref{lem:monoid}, we only need to show that joins exist. So let $u,v\in \mathcal{L}$ be any two elements. Our goal is to show that their join~$u \vee v$ exists, so we can assume that $u$ and $v$ are incomparable. Notice that the only elements in~$\mathcal{L}$ which cover more than one element are $x_1^{m}$ for $m\geq n$. Moreover, these elements are totally ordered. Hence, if $u$ and $v$ have an upper bound in $\mathcal{L}$, they will have a least upper bound (which will be $x_1^{m}$ for some $m\geq n$). So we only need to show that $u$ and $v$ do have an upper bound. But this is also easy to see. Say the length of $u$ is $\ell(u)=m_1$ and the length of $v$ is $\ell(v)=m_2$ and suppose $m_1 \geq m_2$ by symmetry. Then $ux_1^{n-1}=x_1^{n-1+m_1}=vx_1^{n-1+(m_1-m_2)}$ will be an upper bound for $u$ and $v$.

That the core of $\mathcal{L}$ is $L=\{\hat{0}\} \oplus r \cdot [n-1] \oplus \{\hat{1}\}$ is clear, since $x_1^n$ will be the join of the atoms $x_1,\ldots,x_r$. That $F(\mathcal{L};x)^{-1} =\chi^*(L;x)$ is then of course a consequence of \cref{cor:upho_lat_rgf}, and that $F(\mathcal{L};x)^{-1} =\chi^*(L;x) = 1-rx+(r-1)x^n$ is also straightforward. So we are done.
\end{proof}

\subsection{Garside monoids and Coxeter groups}

As mentioned, to produce sophisticated examples of monoids satisfying the conditions of \cref{lem:monoid} we need to employ the theory of Garside monoids. In our discussion of this theory, we generally follow the terminology of~\cite{dehornoy1999gaussian, dehornoy2015foundations}. Let us review this terminology now.

A bounded atomic, finitely generated monoid is called \dfn{right Gaussian} if it is left cancellative and every pair of elements have a least (with respect to $\leq_L$) common right multiple. Dually, it is called \dfn{left Gaussian} if it is right cancellative and every pair of elements have a least (with respect to $\leq_R$) common left multiple. It is called a \dfn{Gaussian monoid} if it is both left and right Gaussian. 

A Gaussian monoid $M$ is called a \dfn{Garside monoid} if there exists a~$\Delta \in M$ for which the set of left divisors of $\Delta$ equals the set of right divisors of $\Delta$, and for which this set is finite and generates $M$. This $\Delta$ is called a \dfn{Garside element} for $M$. In general, a Garside monoid will have multiple Garside elements, but it will have a unique minimal (with respect to either $\leq_L$ or $\leq_R$) one. We assume from now on that $\Delta$ is always chosen to be minimal. This minimal Garside element $\Delta$ is the join (with respect to either $\leq_L$ or $\leq_R$) of the atoms of the Garside monoid $M$. The (left or right) divisors of $\Delta$ are called the \dfn{simple elements} of $M$. Notice that by definition there are finitely many simple elements.

If $G$ is a group and $M$ is a monoid, we say $G$ is a \dfn{group of (left) fractions} of~$M$ if there is a monoid embedding $M \subseteq G$ for which every element of $G$ has the form $h^{-1}g$ with $h,g \in M\subseteq G$. A group $G$ is called a \dfn{Garside group} if it is the group of fractions of a Garside monoid $M$.

For a monoid $M$ belonging to any of the above classes (i.e., left/right Gaussian, Gaussian, or Garside), let us say that $M$ is homogeneous if it is \emph{homogeneously} finitely generated. 

Evidently, the monoids satisfying the conditions of \cref{lem:monoid} are exactly the homogeneous right Gaussian monoids. Thus, we could restrict our attention to those. However, the important examples of monoids coming from finite Coxeter groups are in fact Garside monoids, so we will work with Garside monoids. The class of Garside monoids was introduced because these monoid, and especially their associated Garside groups, enjoy good algorithmic properties, like a solvable word problem. But these algorithmic properties will not concern us here. The point for us is the following, which follows immediately from~\cref{lem:monoid} and the definitions we just went over.

\begin{lemma} \label{lem:garside}
Let $M$ be a homogeneous Garside monoid. Then $\mathcal{L}\coloneqq (M,\leq_L)$ is an upho lattice whose core is its lattice of simple elements.
\end{lemma}

In \cref{lem:garside}, by ``lattice of simple elements,'' we mean the set of simple elements of $M$ under the partial order $\leq_L$, which will form a finite graded lattice.

Our next objective is to review the most important examples of Garside monoids, which come from finite Coxeter groups. Accordingly, we very briefly review the theory of finite Coxeter groups. See~\cite{humphreys1990reflection, bjorner2005combinatorics} for textbook accounts. Since we will work with both groups and monoids, we will be careful to distinguish when a presentation is a group presentation versus a monoid presentation.

A \dfn{Coxeter system} $(W,S)$ is a group $W$ together with a set $S=\{s_1,\ldots,s_r\} \subseteq W$ of generators for which $W = \langle S \mid s_i^2 = 1 \textrm{ for all $i$}, (s_is_j)^{m_{i,j}} = 1 \textrm{ for all $i<j$} \rangle$ for certain integer parameters $m_{i,j} \geq 2$.\footnote{Often $m_{i,j}=\infty$ is allowed as well, but since we are exclusively interested in \emph{finite} Coxeter systems this technicality will not concern us.} The number $r$ of generators in $S$ is called the \dfn{rank} of the Coxeter system. We say the Coxeter system $(W,S)$ is \dfn{finite} if the group~$W$ is finite.

For $(W,S)$ a Coxeter system, the group $W$ is called a \dfn{Coxeter group}, and the set of generators $S$ is called the set of \dfn{simple reflections}. If $g,h\in G$ are two elements of a group $G$ we denote the conjugation of $g$ by $h$ as $g^{h} \coloneqq h^{-1}gh$. The set~$T\coloneqq \{s^w\colon s \in S, w \in W\}\subseteq W$ of $W$-conjugates of $S$ is called the set of \dfn{reflections} of the Coxeter system $(W,S)$. The terminology of ``simple reflections'' and ``reflections'' comes from a geometric picture for Coxeter groups, but we will avoid discussion of this geometric picture here.

For example, the symmetric group $S_n$ of permutations of $[n]$ is a finite Coxeter group of rank $n-1$, with the simple reflections being the adjacent transpositions $s_i = (i,i+1)$ for $i =1,\ldots,n-1$. Here we have $m_{i,j}=3$ for all $1\leq i < j \leq n-1$, and the reflections are all the transpositions $t_{i,j} \coloneqq (i,j)$ for $1 \leq i < j \leq n$.

In fact, all the finite Coxeter systems have been classified. They fall into a number of infinite families together with a few exceptional examples; these are the (Cartan--Killing) types. See~\cite{humphreys1990reflection, bjorner2005combinatorics} for a complete account of the classification. In this classification, the symmetric group $S_n$ is the Type $A_{n-1}$ Coxeter group. We also remark that the finite Coxeter groups are known to be precisely the finite real reflection groups. 

Continue to let $(W,S)$ be a Coxeter system. For $s_i,s_j \in S$ and $m\geq 0$, we write $(s_i,s_j)^{[m]} \coloneqq s_is_js_is_j\cdots$, a word with~$m$ letters. The \dfn{braid group} associated to~$(W,S)$ is the group $\langle S \mid (s_i,s_j)^{[m_{i,j}]} = (s_j,s_i)^{[m_{i,j}]} \textrm{ for all $i<j$} \rangle$. Notice that there is a canonical embedding of a Coxeter group into its braid group. For $W=S_n$, this braid group is usually denoted~$B_n$. The elements of $B_n$ can be represented diagrammatically as braids, i.e., collections of strands connecting $n$ initial points to~$n$ terminal points, and hence the name.

The braid group associated to any finite Coxeter system is a Garside group. In fact, there are \emph{two different ways} to realize such a braid group as a Garside group, i.e., there are two different associated Garside monoids whose group of fractions is the braid group. These two Garside monoids yield two upho lattices, with two different cores, associated to any finite Coxeter system. To conclude this section, we review the construction and combinatorics of these two Garside monoids associated to a finite Coxeter system.

\subsubsection{The weak order and the classical braid monoid}

So now let us fix a finite Coxeter system $(W,S=\{s_1,\ldots,s_r\})$, with parameters $m_{i,j}$ as above. 

The $\dfn{length}$ $\ell(w)$ of an element $w \in W$ is the minimum length of an expression $w=s_{i_1}\cdots s_{i_k}$ for $w$ as a product of the simple reflections $s_i$. The \dfn{(right) weak order} of $W$ is the partial order with cover relations $w \lessdot ws$ whenever $\ell(ws) = \ell(w) + 1$ for $w\in W$ and $s \in S$. This indeed defines a partial order, which can be equivalently described as $u \leq w$ if and only $\ell(u) + \ell(u^{-1}w) = \ell(w)$ for any $u,w \in W$. Weak order of a finite Coxeter group is a finite graded lattice, with rank function $\ell$. But note that the rank of weak order is unrelated to the rank of the Coxeter system: the rank of weak order is $\ell(w_0)$, where $w_0\in W$ is the maximum element under weak order, which is called the \dfn{longest element}. 

See~\cite[\S3]{bjorner2005combinatorics} for all these basic statements about weak order.

\begin{example} \label{ex:weak_s3}
In \cref{fig:weak_s3} we depict the weak order of the symmetric group~$S_3$. In this figure, we have written permutations in one-line notation. We have also labeled each cover relation with the corresponding simple reflection.
\end{example}

\begin{figure}
\begin{tikzpicture}
\node (1) at (0,0) {$123$};
\node (2) at (-1,1) {$213$};
\node (3) at (1,1) {$132$};
\node (4) at (-1,2) {$231$};
\node (5) at (1,2) {$312$};
\node (6) at (0,3) {$321$};
\draw (1) -- (2) node[midway,left] {$s_1$};
\draw (1) -- (3) node[midway,right] {$s_2$};
\draw (2) -- (4) node[midway,left] {$s_2$};
\draw (3) -- (5) node[midway,right] {$s_1$};
\draw (4) -- (6) node[midway,left] {$s_1$};
\draw (5) -- (6) node[midway,right] {$s_2$};
\end{tikzpicture}
\caption{The weak order of the symmetric group $S_3$.} \label{fig:weak_s3}
\end{figure}
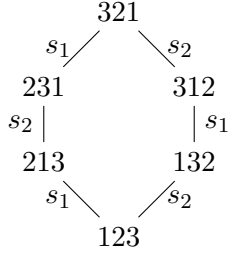

We now describe how to realize the weak order of any finite Coxeter group as a core of an upho lattice. Let $\mathbf{S} = \{\mathbf{s}_1,\ldots,\mathbf{s}_r\}$ be a collection of letters corresponding to the simple reflections $S=\{s_1,\ldots,s_r\}$. For $\mathbf{s}_i, \mathbf{s}_j \in \mathbf{S}$ and $m\geq 0$, let us again write $(\mathbf{s}_i,\mathbf{s}_j)^{[m]} \coloneqq \mathbf{s}_i\mathbf{s}_j\mathbf{s}_i\mathbf{s}_j\cdots $, a word with $m$ letters. The \dfn{classical braid monoid} for~$(W,S)$ is the monoid~$\langle \mathbf{S} \mid (\mathbf{s}_i,\mathbf{s}_j)^{[m_{i,j}]} = (\mathbf{s}_j,\mathbf{s}_i)^{[m_{i,j}]} \textrm{ for $i < j$}\rangle$. 

The point of this construction is the following:

\begin{thm}[{See \cite[Ch.~IX, \S1]{dehornoy2015foundations}}] \label{thm:classical_braid}
For any finite Coxeter system $(W,S)$, its classical braid monoid is a (homogeneous) Garside monoid, whose group of fractions is the corresponding braid group, and whose lattice of simple elements is isomorphic to the weak order of $W$.
\end{thm}

\Cref{thm:classical_braid} traces back to work of Garside~\cite{garside1969braid}, in Type~A, and of Brieskorn--Saito~\cite{brieskorn1972artin} and Deligne~\cite{deligne1972immeubles}, in all finite types. For more on the history of this theorem, consult~\cite{dehornoy2015foundations}. \Cref{thm:classical_braid}, together with \cref{lem:garside}, implies the following.

\begin{cor}
For any finite Coxeter system $(W,S)$, the weak order of $W$ is the core of an upho lattice (namely, the corresponding classical braid monoid).
\end{cor}

\begin{example}
In \cref{fig:braid_s3} we depict the classical braid monoid for the symmetric group~$S_3$. For ease of display, we have written $a=\mathbf{s}_1$ and $b=\mathbf{s}_2$ in this figure. Observe how the weak order from \cref{ex:weak_s3} is the core of this upho lattice.
\end{example}

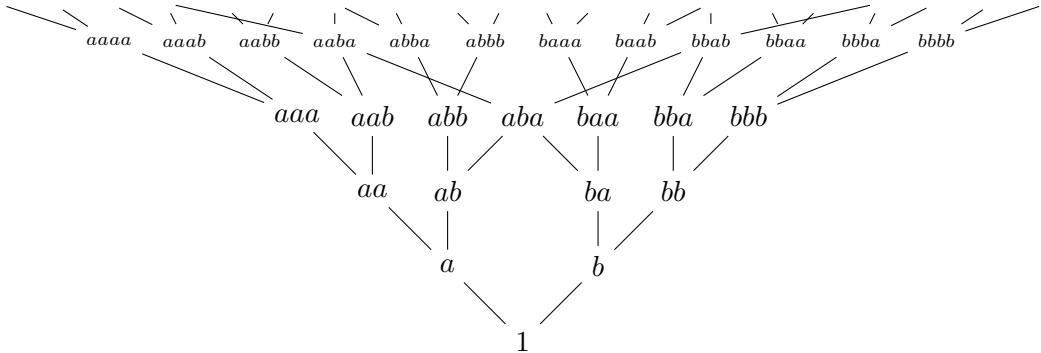
\begin{figure}
\begin{tikzpicture}
\node (A) at (0,0) {$1$};
\node (B) at (-1,1) {$a$};
\node (C) at (1,1) {$b$};
\node (D) at (-2,2) {$aa$};
\node (E) at (-1,2) {$ab$};
\node (F) at (1,2) {$ba$};
\node (G) at (2,2) {$bb$};
\node (H) at (-3,3) {$aaa$};
\node (I) at (-2,3) {$aab$};
\node (J) at (-1,3) {$abb$};
\node (K) at (0,3) {$aba$};
\node (L) at (1,3) {$baa$};
\node (M) at (2,3) {$bba$};
\node (N) at (3,3) {$bbb$};
\node (P) at (-5.5,4) {\tiny $aaaa$};
\node (Q) at (-4.5,4) {\tiny $aaab$};
\node (R) at (-3.5,4) {\tiny $aabb$};
\node (S) at (-2.5,4) {\tiny $aaba$};
\node (T) at (-1.5,4) {\tiny $abba$};
\node (U) at (-0.5,4) {\tiny $abbb$};
\node (V) at (0.5,4) {\tiny $baaa$};
\node (W) at (1.5,4) {\tiny $baab$};
\node (X) at (2.5,4) {\tiny $bbab$};
\node (Y) at (3.5,4) {\tiny $bbaa$};
\node (Z) at (4.5,4) {\tiny $bbba$};
\node (!) at (5.5,4) {\tiny $bbbb$};
\node (AA) at (-7,4.5) {}; 
\node (BB) at (-6.25,4.5) {}; 
\node (CC) at (-5.5,4.5) {}; 
\node (DD) at (-4.75,4.5) {}; 
\node (EE) at (-4,4.5) {}; 
\node (FF) at (-3.25,4.5) {}; 
\node (GG) at (-2.5,4.5) {}; 
\node (HH) at (-1.75,4.5) {}; 
\node (II) at (-1,4.5) {}; 
\node (JJ) at (-0.25,4.5) {}; 
\node (KK) at (0.25,4.5) {}; 
\node (LL) at (1,4.5) {}; 
\node (MM) at (1.75,4.5) {}; 
\node (NN) at (2.5,4.5) {}; 
\node (OO) at (3.25,4.5) {}; 
\node (PP) at (4,4.5) {}; 
\node (QQ) at (4.75,4.5) {}; 
\node (RR) at (5.5,4.5) {}; 
\node (SS) at (6.25,4.5) {}; 
\node (TT) at (7,4.5) {}; 
\draw (B) -- (A) -- (C);
\draw (D) -- (B) -- (E);
\draw (F) -- (C) -- (G);
\draw (H) -- (D) -- (I);
\draw (J) -- (E) -- (K) -- (F) -- (L);
\draw (M) -- (G) -- (N);
\draw (P) -- (H) -- (Q);
\draw (R) -- (I) -- (S) -- (K) -- (X) -- (M) -- (Y);
\draw (T) -- (J) -- (U);
\draw (V) -- (L) -- (W);
\draw (Z) -- (N) -- (!);
\draw (AA) -- (P) -- (BB);
\draw (CC) -- (Q) -- (DD) -- (S) -- (GG) -- (T) -- (HH);
\draw (EE) -- (R) -- (FF);
\draw (II) -- (U) -- (JJ);
\draw (KK) -- (V) -- (LL);
\draw (OO) -- (Y) -- (PP);
\draw (MM) -- (W) -- (NN) -- (X) -- (QQ) -- (Z) -- (RR);
\draw (SS) -- (!) -- (TT);
\end{tikzpicture}
\caption{The classical braid monoid $\langle a,b\mid aba = bab\rangle$ for the symmetric group $S_3$.} \label{fig:braid_s3}
\end{figure}

\begin{remark}
There is no product formula for the characteristic polynomial of the weak order of a finite Coxeter group~$W$, although there is an expression as a sum over ``parabolic longest elements.'' See, e.g., the recent paper~\cite{kim2024characteristic} for discussion of the characteristic polynomial of weak order.
\end{remark}

\subsubsection{The noncrossing partition lattice and the dual braid monoid}

Continue to fix a finite Coxeter system $(W,S)$ as above. There is a ``dual'' approach to Coxeter systems, where the role of the simple reflections~$S$ is replaced by \emph{all} the reflections~$T$. For background on the combinatorics associated to this dual approach, see the monograph~\cite{armstrong2009generalized}. 

The \dfn{absolute length} $\ell_T(w)$ of an element $w\in W$ is the minimum length of an expression $w=t_{1}\cdots t_{k}$ for $w$ as a product of reflections $t_i \in T$. The \dfn{absolute order} of~$W$ is the partial order with cover relations $w \lessdot wt$ whenever $\ell_T(wt) = \ell_T(w) + 1$ for $w \in W$ and $t\in T$. This indeed defines a partial order on $W$, which can be equivalently described as $u \leq w$ if and only $\ell_T(u) + \ell_T(u^{-1}w) = \ell_T(w)$ for any $u,w \in W$.  Absolute order of a finite Coxeter group is a finite graded poset, with rank function $\ell_T$, but it is not a lattice because it has multiple maximal elements. Indeed, all Coxeter elements are maximal elements in absolute order, where we recall that a \dfn{Coxeter element} $c=s_{i_1}s_{i_2}\cdots s_{i_r}$ is a product of all the simple reflections, each appearing once, in any order. 

But, the interval $[e,c]$ in absolute order from the identity element $e \in W$ to any fixed Coxeter element $c \in W$ \emph{is} a finite graded lattice. Moreover, the isomorphism type of this lattice does not depend on the choice of Coxeter element. This lattice is called the \dfn{noncrossing partition lattice} of $W$. Notice that the rank of the noncrossing partition lattice of $W$ is the same as the rank of the Coxeter system $(W,S)$. 

See~\cite[\S2.4 and 2.6]{armstrong2009generalized} for these basic facts about absolute order and noncrossing partition lattices.

\begin{figure}
\begin{tikzpicture}
\node (1) at (0,0) {\color{red} $\mathbf{e}$};
\node (2)  at (-2,1) {\color{red} $\mathbf{(1,2)}$};
\node (3) at (0,1) {\color{red} $\mathbf{(2,3)}$};
\node (4) at (2,1) {\color{red} $\mathbf{(1,3)}$};
\node (5) at (-1,2) {\color{red} $\mathbf{(1,2,3)}$};
\node (6) at (1,2) {$(1,3,2)$};
\draw[red,ultra thick] (1)--(2)--(5)--(3)--(1)--(4)--(5);
\draw (2)--(6)--(3)--(6)--(4);
\end{tikzpicture}
\caption{The absolute order of the symmetric group $S_3$. In red and bold is the noncrossing partition lattice of $S_3$.} \label{fig:noncrossing}
\end{figure}
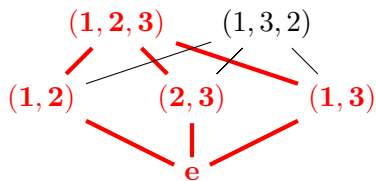

\begin{example} \label{ex:absolute_s3}
The absolute order of the symmetric group $S_3$ is depicted in \cref{fig:noncrossing}. Here we have written permutations using cycle notation, with $e$ denoting the identity. Also, in bold and in red is the noncrossing partition lattice of $S_3$, for the choice of Coxeter element $c=(1,2,3)$.
\end{example}

The noncrossing partition lattice for the symmetric group $S_n$ is isomorphic to the restriction of the partition lattice $\Pi_n$ to the set of noncrossing partitions. The isomorphism is given by the cycle decomposition. Here we say two subsets~$A,B\subseteq [n]$ are crossing if there are $i,j \in A$, $k,l \in B$ with $i < k < j < l$, and we say that a partition~$\pi$ of $[n]$ is noncrossing if no two of its blocks are crossing. This explains the name ``noncrossing partition lattice.'' See~\cite{simion2000noncrossing} for a survey on the classical (i.e., Type~A) noncrossing partition lattice.

We now describe how to realize the noncrossing partition lattice of any finite Coxeter group as a core of an upho lattice. Let $\mathbf{T} = \{\mathbf{t}\colon t \in T\}$ be a collection of letters corresponding to all the reflections. For $\mathbf{s}, \mathbf{t}\in T$ we write $\mathbf{t}^{\mathbf{s}}$ for the letter corresponding to the conjugate $t^s\in T$. The \dfn{dual braid monoid} associated to $(W,S)$ is the monoid $\langle \mathbf{T} \mid \mathbf{t}\mathbf{s} = \mathbf{s}\mathbf{t}^{\mathbf{s}} \textrm{ for all $\mathbf{s}\neq\mathbf{t} \in \mathbf{T}$}\rangle$. 

The point of this construction is the following:

\begin{thm}[{See \cite[Ch.~IX, \S2]{dehornoy2015foundations}}] \label{thm:dual_braid}
For any finite Coxeter system $(W,S)$, its dual braid monoid is a (homogeneous) Garside monoid, whose group of fractions is the corresponding braid group, and whose lattice of simple elements is isomorphic to the noncrossing partition lattice of $W$.
\end{thm}

\Cref{thm:dual_braid} traces back to work of Birman--Ko--Lee~\cite{birman1998new}, in Type~A, and of Brady--Watt~\cite{brady2002artin} and Bessis~\cite{bessis2003dual}, in all finite types. Again, for more on the history, consult~\cite{dehornoy2015foundations} or \cite{armstrong2009generalized}. \Cref{thm:dual_braid}, together with \cref{lem:garside}, implies the following.

\begin{cor}
For any finite Coxeter system $(W,S)$, the noncrossing partition lattice $W$ is the core of an upho lattice (namely, the corresponding dual braid monoid).
\end{cor}

\begin{example}
In \cref{fig:monoid} we depict the noncrossing partition lattice for the symmetric group~$S_3$. For ease of display, we have written $a=\mathbf{t_{1,2}}$, $b=\mathbf{t_{2,3}}$, and $c=\mathbf{t_{1,3}}$ in this figure. Observe how the noncrossing partition lattice from \cref{ex:absolute_s3} is the core of this upho lattice. We also note that the noncrossing partition lattice of $S_3$ happens to be isomorphic to the partition lattice $\Pi_3$, since no two subsets of $\{1,2,3\}$ are crossing. Hence, \cref{fig:monoid} and \cref{fig:partitions} depict two different upho lattices with isomorphic cores.
\end{example}

\begin{remark}
There is no product formula for the characteristic polynomial of the noncrossing partition lattice of a finite Coxeter group~$W$, although there is an expression as a sum over the faces of the ``positive part of the cluster complex.'' See, e.g., the recent paper~\cite{josuatverges2023koszulity} for discussion of the characteristic polynomial of noncrossing partition lattices.
\end{remark}

\section{Obstructions for cores of upho lattices} \label{sec:obstruct}

In this section we explore techniques for showing that a given finite graded lattice is \emph{not} the core of any upho lattice. The relationship between the rank generating function of an upho lattice and the characteristic polynomial of its core implies some characteristic polynomial obstructions for being a core. There are also structural obstructions, which say that a finite lattice must be partly self-similar to be a core. These obstructions let us to show that many well-studied lattices are not cores.

\subsection{Characteristic polynomial obstructions} 

The following is an immediate corollary of \cref{cor:upho_lat_rgf}.

\begin{lemma} \label{lem:positivity}
Let $L$ be a finite graded lattice which is the core of some upho lattice. Then the coefficients of the formal power series $\chi^*(L;x)^{-1}$ are all nonnegative.
\end{lemma}

Already \cref{lem:positivity} lets us rule out many lattices as cores. The following is another similar, but slightly different, characteristic polynomial obstruction.

\begin{lemma} \label{lem:positivity2}
Let $L$ be a finite graded lattice which is the core of some upho lattice. Then, for any integer $m \geq 1$, the coefficients of the formal power series $\chi^*(L;x^m)/\chi^*(L;x)^m$ are all nonnegative.
\end{lemma}

\begin{proof}
In~\cite[Corollary 14]{hopkins2022note} we proved that for an upho lattice $\mathcal{L}$ with core $L$,
\[\sum_{\substack{(p_1,\ldots,p_m) \in \mathcal{L}^m,\\ p_1 \wedge \cdots \wedge p_m = \hat{0}}}x^{\rho(p_1) + \cdots + \rho(p_m)} = \frac{\chi^*(L;x^m)}{\chi^*(L;x)^m}.\]
The left side of this equation evidently has all nonnegative coefficients, so the right must as well.
\end{proof}

\begin{remark}
Joel Lewis pointed out to me that the proof of \cite[Corollary 14]{hopkins2022note} in fact shows
\[\sum_{\substack{(p_1,\ldots,p_m) \in \mathcal{L}^m,\\ p_1 \wedge \cdots \wedge p_m = \hat{0}}}x_1^{\rho(p_1)}x_2^{\rho(p_2)}\cdots x_m^{\rho(p_m)} = \frac{\chi^*(L;x_1x_2\cdots x_m)}{\chi^*(L;x_1)\chi^*(L;x_2) \cdots \chi^*(L;x_m)}.\]
So, we can say more strongly that, for any $m \geq 1$, the coefficients of the multivariate formal power series $\chi^*(L;x_1x_2\cdots x_m)/(\chi^*(L;x_1)\chi^*(L;x_2) \cdots \chi^*(L;x_m))$ must all be nonnegative for any finite graded lattice $L$ which is a core.
\end{remark}

In the remainder of this subsection, we use these characteristic polynomial obstructions to show that various finite graded lattices are not cores.

\subsubsection{Face lattices of polytopes}

Recall that a polytope $P$ is the convex hull of finitely many points in Euclidean space $\mathbb{R}^n$. A face of $P$ is a subset of points $F\subseteq P$ which maximizes some linear functional. The dimension $\mathrm{dim}(F)$ of a face $F$ is the dimension of its affine span. By convention, the empty set $\varnothing$ is also considered a face, called the empty face, with $\mathrm{dim}(\varnothing)=-1$. The \dfn{face lattice} $L(P)$ of $P$ is the poset of faces of $P$ partially ordered by containment. The face lattice is a finite graded lattice, with the rank of a face $F$ being $\mathrm{dim}(F)+1$.

For these basic facts about polytopes and their faces, see~\cite[Ch.~2]{ziegler1995polytopes}.

Any face of a polytope $P$ is a convex hull of vertices (zero-dimensional faces). So we can identify the face lattice $L(P)$ with a poset of certain subsets of vertices ordered by containment. 

\begin{example}
\Cref{fig:square} depicts the face lattice of a square whose vertices are $a,b,c,d$ in clockwise order. In this figure, the faces are represented by subsets of these vertices, and the sets are written in a compact way without braces or commas.
\end{example}

\begin{figure}
\begin{tikzpicture}
\node (A) at (0,0) {$\varnothing$};
\node (B) at (-1.5,1) {$a$};
\node (C) at (-0.5,1) {$b$};
\node (D) at (0.5,1) {$d$};
\node (E) at (1.5,1) {$c$};
\node (F) at (-1.5,2) {$ab$};
\node (G) at (-0.5,2) {$ad$};
\node (H) at (0.5,2) {$bc$};
\node (I) at (1.5,2) {$cd$};
\node (J) at (0,3) {$abcd$};
\draw (B) -- (A) -- (C) -- (A) -- (D) -- (A) -- (E);
\draw (B) -- (F) -- (C) -- (H) -- (E) -- (I) -- (D) -- (G) -- (B);
\draw (F) -- (J) -- (G) -- (J) -- (H) -- (J) -- (I);
\end{tikzpicture}
\caption{The face lattice of a square.} \label{fig:square}
\end{figure}
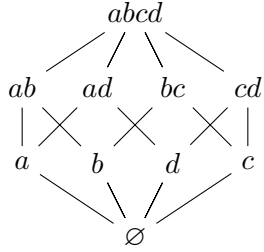

Notice that the face lattice of a $n$-dimensional simplex is isomorphic to the Boolean lattice $B_{n+1}$, which we know is a core. Hence, it is reasonable to ask which other face lattices of polytopes are cores. We do not know if the face lattice of a square is a core. But, going up a dimension, we can show that the face lattice of  an octahedron is not a core.

\begin{prop} \label{prop:octa}
Let $P$ be an octahedron. Then its face lattice $L(P)$ is not the core of any upho lattice.
\end{prop}
\begin{proof}
Using SageMath, we compute that $\chi^{*}(L(P);x) = 1-6x+12x^2-8x^3+x^4$ and that $[x^{13}] \chi^{*}(L(P);x)^{-1} = -123704$, where $[x^n]F(x)$ means the coefficient of $x^n$ in the power series $F(x)$. Hence, by \cref{lem:positivity}, $L(P)$ is not a core.
\end{proof}

See \cref{subsec:cross_hyper} below for an extension of \cref{prop:octa} to all cross-polytopes and hypercubes of dimension three or greater.

\subsubsection{Bond lattices of graphs}

Let $G$ be a connected, simple graph on the vertex set~$[n]$. A partition $\pi$ of $[n]$ is called $G$-connected if the induced subgraph of $G$ on each block of $\pi$ remains connected. The \dfn{bond lattice} $L(G)$ of $G$ is the restriction of the partition lattice~$\Pi_n$ to the $G$-connected partitions. The bond lattice $L(G)$ is a finite graded lattice; in fact, it is a geometric lattice, corresponding to the graphic matroid of~$G$. We have that $\chi^{*}(L(G);x) = x^n \cdot \chi(G;x^{-1})$ where $\chi(G;x)$ is the chromatic polynomial of $G$.

For these basic facts about bond lattices, see, e.g.,~\cite[\S2.3]{stanley2007hyperplane}. 

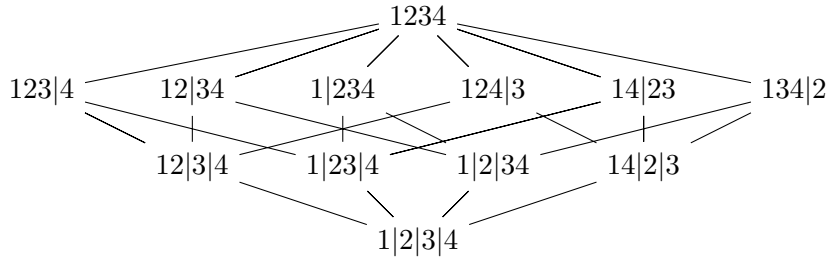
\begin{figure}
\begin{tikzpicture}
\node (A) at (0,0) {$1|2|3|4$};
\node (B) at (-3,1) {$12|3|4$};
\node (C) at (-1,1) {$1|23|4$};
\node (D) at (1,1) {$1|2|34$};
\node (E) at (3,1) {$14|2|3$};
\node (F) at (-5,2) {$123|4$};
\node (G) at (-3,2) {$12|34$};
\node (H) at (-1,2) {$1|234$};
\node (I) at (1,2) {$124|3$};
\node (J) at (3,2) {$14|23$};
\node (K) at (5,2) {$134|2$};
\node (L) at (0,3) {$1234$};
\draw (B) -- (A) -- (C) -- (A) -- (D) -- (A) -- (E);
\draw (F) -- (B) -- (G) -- (D) -- (K) -- (E) -- (J) -- (C) -- (F) -- (B) -- (I) -- (E) -- (J) -- (C) -- (H) -- (D);
\draw (F) -- (L) -- (G) -- (L) -- (H) -- (L) -- (I) -- (L) -- (J) -- (L) -- (K);
\end{tikzpicture}
\caption{The bond lattice of the cycle graph on four vertices.} \label{fig:c4}
\end{figure}

\begin{example}
For any $n \geq 1$, the cycle graph $C_n$ has vertex set~$[n]$ and edges $\{1,2\}, \{2,3\}, \ldots,\{n-1,n\}, \{n,1\}$. The bond lattice $L(C_4)$ is depicted in \cref{fig:c4}.
\end{example}

Notice that, for any $n \geq 1$, $\Pi_n=L(K_n)$ where $K_n$ is the complete graph on $[n]$. Since we know that $\Pi_n$ is a core, it is reasonable to ask which other graphs have bond lattices that are cores. We can show that the bond lattice of $C_4$ is not a core.

\begin{prop} \label{prop:c4}
The bond lattice $L(C_4)$ of the cycle graph $C_4$ on four vertices is not the core of any upho lattice.
\end{prop}
\begin{proof}
Using SageMath, we compute that $\chi^{*}(L(C_4);x) = 1-4x+6x^2-3x^3$ and that $[x^7]\chi^{*}(L(C_4);x)^{-1} = -80$. So by \cref{lem:positivity}, $L(C_4)$ is not a core.
\end{proof}

See \cref{subsec:uniform_matroid} below for an extension of \cref{prop:c4} to all cycle graphs~$C_n$ for $n \geq 4$. In light of these (non-)examples, one reasonable guess is that the bond lattices which are cores are precisely those which are supersolvable. In~\cite{stanley1972supersolvable} Stanley proved that the supersolvable bond lattices are the ones coming from chordal graphs, where a graph is chordal if it has no induced $C_n$ for $n \geq 4$.

Before we move on to the structural obstructions, we might wonder if there is any lattice which can be proved to not be a core using \cref{lem:positivity2} instead of \cref{lem:positivity}. Let us give such an example.

\begin{figure}
\begin{tikzpicture}[scale=0.75]
\node[circle,fill,inner sep=1.5pt] (1) at (0,0) {};
\node[circle,fill,inner sep=1.5pt] (2) at (1,0.5) {};
\node[circle,fill,inner sep=1.5pt] (3) at (1,-0.5) {};
\node[circle,fill,inner sep=1.5pt] (4) at (2,0.5) {};
\node[circle,fill,inner sep=1.5pt] (5) at (2,-0.5) {};
\node[circle,fill,inner sep=1.5pt] (6) at (3,0.5) {};
\node[circle,fill,inner sep=1.5pt] (7) at (3,-0.5) {};
\node[circle,fill,inner sep=1.5pt] (8) at (4,0.5) {};
\node[circle,fill,inner sep=1.5pt] (9) at (4,-0.5) {};
\node[circle,fill,inner sep=1.5pt] (10) at (5,0.5) {};
\node[circle,fill,inner sep=1.5pt] (11) at (5,-0.5) {};
\node[circle,fill,inner sep=1.5pt] (12) at (6,0.5) {};
\node[circle,fill,inner sep=1.5pt] (13) at (6,-0.5) {};
\node[circle,fill,inner sep=1.5pt] (14) at (7,0.5) {};
\node[circle,fill,inner sep=1.5pt] (15) at (7,-0.5) {};
\node[circle,fill,inner sep=1.5pt] (16) at (8,0) {};
\draw (2) -- (1) -- (3) -- (2) -- (4) -- (5) -- (3) -- (5) -- (7) -- (6) -- (4) -- (6) -- (8) -- (9) -- (7) -- (9) -- (11) -- (10) -- (8) -- (10) -- (12) -- (13) -- (11) -- (13) -- (15) -- (14) -- (12) -- (14) -- (16) -- (15);
\end{tikzpicture}
\caption{The graph from \cref{prop:pos2_ex}.} \label{fig:pos2_ex}
\end{figure}

\begin{prop} \label{prop:pos2_ex}
Let $G$ be the graph on $16$ vertices depicted in \cref{fig:pos2_ex}. Then its bond lattice $L(G)$ is not the core of any upho lattice.
\end{prop}

\begin{proof}
Using SageMath, we compute $\chi^{*}(L(G);x) = (1-x)(1-2x)^2(1-3x+3x^2)^6$. We also compute that $\chi^{*}(L(G);x)^{-1}$ has all positive coefficients, so \cref{lem:positivity} does not apply. However, $[x^{24}]\chi^{*}(L(G);x^2)/\chi^{*}(L(G);x)^2 = -269758375958758$, so by \cref{lem:positivity2} with $m=2$, $L(G)$ is not a core.
\end{proof}

\subsection{Structural obstructions} 

The following proposition is an immediate consequence of the definition of the core of an upho lattice, but it is worth recording since it rules out many finite graded lattices as cores. (For instance, it implies the only \emph{distributive} lattices which are cores are the Boolean lattices~$B_n$; see~\cite[\S3.4]{stanley2012ec1}.)

\begin{prop} \label{prop:max}
Let $L$ be a finite graded lattice which is the core of some upho lattice. Then its maximum $\hat{1} \in L$ is the join of its atoms $s_1,\ldots,s_r \in L$.
\end{prop}

In fact, using \cref{prop:max} we can show that the collection of finite lattices which are cores of upho lattices is not closed under duality.

\begin{example} \label{ex:dual}
Let $L$ be the finite lattice depicted in \cref{fig:dual}. Its dual $L^*$ cannot be the core of any upho lattice due to \cref{prop:max}. But Joel Lewis explained to me that $L$ itself actually \emph{is} a core. I thank him for letting me include this explanation here. Define the monoid $M \coloneqq \langle a,b,c \mid aa = bb, ba = ca\rangle$. It can be shown (e.g., using the techniques of~\cite{hopkins2024upho2}) that $M$ satisfies the conditions of \cref{lem:monoid}, so that $\mathcal{L}\coloneqq (M,\leq_L)$ is an upho lattice, whose core is $L$.
\end{example}

\begin{figure}
\begin{tikzpicture}
\node[circle,fill,inner sep=1.5pt] (1) at (0,0) {};
\node[circle,fill,inner sep=1.5pt] (2) at (-1,1) {};
\node[circle,fill,inner sep=1.5pt] (3) at (0,1) {};
\node[circle,fill,inner sep=1.5pt] (4) at (1,1) {};
\node[circle,fill,inner sep=1.5pt] (5) at (-0.5,2) {};
\node[circle,fill,inner sep=1.5pt] (6) at (0.5,2) {};
\node[circle,fill,inner sep=1.5pt] (7) at (0,3) {};
\draw (2) -- (1) -- (3) -- (1) -- (4);
\draw (2) -- (5) -- (3) -- (6) -- (4);
\draw (5) -- (7) -- (6);
\end{tikzpicture}
\caption{The lattice $L$ from \cref{ex:dual} which is a core but whose dual $L^*$ is not a core.} \label{fig:dual}
\end{figure}
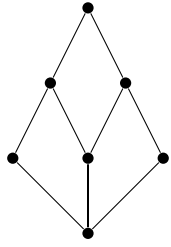

\Cref{prop:max} says something about the join of the elements covering $\hat{0}$ in a core~$L$. Looking at the join of the elements covering an arbitrary element $x\in L$ is a good idea, and leads to other, nontrivial obstructions for cores. The following lemma says that a core must already be ``partly self-similar'' in order to fit into an upho lattice.

\begin{lemma} \label{lem:structural}
Let $L$ be a finite graded lattice which is the core of some upho lattice. Let $x \in L\setminus \{\hat{0},\hat{1}\}$, and let $y_1,\ldots,y_k\in L$ be the elements covering $x$. Then there is a rank-preserving embedding of the interval $[x,y_1\vee \cdots \vee y_k]$ into $L$. 
\end{lemma} 

Recall that ``rank-preserving embedding'' means a map $\varphi\colon [x,y_1\vee \cdots \vee y_k] \to L$ which is an isomorphism onto its image, and which satisfies $\rho(\varphi(z)) = \rho(z)$ for all~$z$. Notice in particular that under the embedding from \cref{lem:structural}, $x$ is mapped to the minimum~$\hat{0} \in L$ and $y_1,\ldots,y_k$ are mapped to atoms $s_1,\ldots,s_k\in L$.

\begin{proof}[Proof of \cref{lem:structural}]
Let $\mathcal{L}$ be an upho lattice whose core is $L$. By assumption, there is some isomorphism $\varphi\colon V_x \to \mathcal{L}$, where $V_x \subseteq \mathcal{L}$ is the principal order filter of~$\mathcal{L}$ generated by $x$. The desired embedding of $[x,y_1\vee \cdots \vee y_k]$ into $L$ is then given by the restriction of $\varphi$ to $[x,y_1\vee \cdots \vee y_k] \subseteq L$.
\end{proof}

In the remainder of this subsection, we use \cref{lem:structural} to show that various lattices are not cores. It appears that this structural obstruction is stronger than the characteristic polynomial obstructions from the previous subsection, although we do not know any precise statement to that effect.

\subsubsection{The face lattices of the cross polytope and hypercube} \label{subsec:cross_hyper}

Recall that the $n$-dimensional \dfn{cross polytope} is the convex hull of all permutations of vectors of the form $(\pm1,0,0,\ldots,0)\in \mathbb{R}^n$. The two dimensional cross polytope is the square and the three dimensional cross polytope is the octahedron.

\begin{thm}
Let $P$ be the $n$-dimensional cross polytope, for $n \geq 3$. Then its face lattice $L(P)$ is not the core of any upho lattice.
\end{thm}

\begin{proof}
Concretely, the face lattice $L(P)$ can be represented as the poset on words of length $n$ in the alphabet $\{0,+,-\}$, with $w\leq w'$ if $w'$ is obtained from $w$ by changing some $0$'s to $\pm$'s, together with a maximum element $\hat{1}$. Let $x$ be any atom of $L(P)$. Then the join of the elements covering $x$ is $\hat{1}$. And  from the concrete description of~$L(P)$, we see that $\hat{1}$ covers $2^{n-1}$ elements in the interval $[x,\hat{1}]$. However, any coatom of $L(P)$ covers only $n$ elements. Since $n \geq 3$, we have $2^{n-1} > n$. Hence, for this choice of $x$, there can be no embedding of the kind required by \cref{lem:structural}, and so $L(P)$ is not a core.
\end{proof}

Recall that the $n$-dimensional \dfn{hypercube} is the convex hull of all vectors of the form $(\pm1, \pm1,\ldots, \pm1) \in \mathbb{R}^n$. The two dimensional hypercube is the square and the three dimensional hypercube is the usual cube. The hypercube is the polar dual of the cross polytope, meaning that their face lattices are dual to one another. However, we have seen in \cref{ex:dual} that this does not directly imply anything about whether the face lattice of the hypercube is a core. Nevertheless, we can also prove the following. 

\begin{thm}
Let $P$ be the $n$-dimensional hypercube, for $n \geq 3$. Then its face lattice $L(P)$ is not the core of any upho lattice.
\end{thm}

\begin{proof}
Let $x$ be any atom of $L(P)$. Then the join of the elements covering $x$ is $\hat{1}$, and the interval $[x,\hat{1}]$ is isomorphic to the face lattice of an $(n-1)$-dimensional simplex, i.e., to the Boolean lattice $B_n$. On the other hand, if~$t$ is any coatom of~$L(P)$, then the interval $[\hat{0},t]$ is isomorphic to the face lattice $L(P')$, where $P'$ is the $(n-1)$-dimensional hypercube. 

So the question is: can there be an embedding of $B_n$ into $L(P')$? And the answer is no, as long as $n \geq 3$. This is because for any $n$ atoms in $L(P')$, there must be at least two of these atoms whose join is the maximum (which has rank $n$), but the join of any two atoms in $B_n$ has rank~$2$. 

Hence, for this choice of $x$, there can be no embedding of the kind required by \cref{lem:structural}, and so $L(P)$ is not a core.
\end{proof}

\subsubsection{The lattice of flats of the uniform matroid} \label{subsec:uniform_matroid}

As mentioned, geometric lattices are exactly the same thing as lattices of flats of matroids. We do not want to review the entire theory of matroids here, but we will discuss one important family of matroids: the uniform matroids.\footnote{Do not confuse ``uniform'' here with the ``uniform sequences'' of lattices from \cref{subsec:uniform_seqs}.}

Concretely, for $2\leq k\leq n$, the lattice of flats of the \dfn{uniform matroid} $U(k,n)$ is obtained from the Boolean lattice~$B_n$ by removing all elements of rank $k$ or higher, and then adjoining a maximal element~$\hat{1}$. So the lattice of flats of $U(2,n)$ is the rank two lattice $M_n$, which we have seen is a core. And the lattice of flats of $U(n,n)$ is the Boolean lattice~$B_n$, which we have seen is a core. But in all other cases, the lattice of flats of the uniform matroid is not a core.

\begin{thm} \label{thm:uni}
For any $2 < k < n$, the lattice of flats of the uniform matroid $U(k,n)$ is not the core of any upho lattice.
\end{thm}

\begin{proof}
Let $L$ be the lattice of flats of the uniform matroid $U(k,n)$ and let $x\in L$ be an atom. Then the join of the elements covering $x$ is $\hat{1}$. And $\hat{1}$ covers $\binom{n-1}{k-2}$ elements in the interval $[x,\hat{1}]$. However, any coatom in $L$ covers only $k-1$ elements. Since $2 < k < n$, we have $\binom{n-1}{k-2} > k-1$. Hence, for this choice of $x$, there can be no embedding of the kind required by \cref{lem:structural}, and so $L$ is not a core.
\end{proof}

\begin{cor}
For any $n \geq 4$, the bond lattice $L(C_n)$ of the cycle graph $C_n$ is not the core of any upho lattice.
\end{cor}

\begin{proof}
The bond lattice $L(C_n)$ is isomorphic to the lattice of flats of $U(n-1,n)$, so this follows immediately from \cref{thm:uni}.
\end{proof}

\bibliography{upho_lattices}{}
\bibliographystyle{abbrv}

\end{document}